\documentclass{amsart}

\usepackage{amsaddr}
\usepackage{graphicx}
\usepackage{epstopdf}
\usepackage{dsfont}
\usepackage{xcolor}
\usepackage{hyperref}
\newtheorem{theorem}{Theorem}[section]
\newtheorem{lemma}[theorem]{Lemma}
\newtheorem{proposition}[theorem]{Proposition}
\newtheorem{corollary}[theorem]{Corollary}
\theoremstyle{definition}
\newtheorem{defi}[theorem]{Definition}

\theoremstyle{remark}
\newtheorem{remark}[theorem]{Remark}

\numberwithin{equation}{section}

\newcommand{\abs}[1]{\lvert#1\rvert}

\newcommand{\fonction}[5]{#1 :\left\{\begin{array}{rcl}
#2 & \longrightarrow & #3\\
#4 & \longmapsto     & #5\\
\end{array} \right.}
\newcommand{\ma}[1]{\mathbb #1}
 
\newcommand{\mc}[1]{\mathcal #1}

\begin{document}
\title{Flat traces for a random partially expanding map}


\date\today
\author{Luc Gossart}
\address{Institut Fourier, 100, rue des maths BP74 38402 Saint-Martin d'Heres France}

\maketitle
\begin{abstract}
    We consider the skew-product of an expanding map $E$ on the circle $\ma T$  with an almost surely $\mc C^k$ random perturbation $\tau=\tau_0+\delta\tau$ of a deterministic function $\tau_0$:
    \[\fonction{F}{\ma T \times \ma R}{\ma T \times \ma R}{(x,y)}{(E(x), y+\tau(x))} .\] 
    The associated transfer operator $\mc L:u \in C^k (\ma T \times \ma R) \mapsto u\circ F$ can be decomposed with respect to frequency in the $y$ variable into a family of operators acting on functions on the circle:
    \[\fonction{\mc L_{\xi}}{C^k(\ma T)}{C^k(\ma T)}{u}{e^{i\xi\tau}u\circ E}.\]
    We show that the flat traces of $\mc L^n_{\xi}$ behave as normal distributions  in the semiclassical limit $n, \xi\to\infty$ up to the Ehrenfest time $n\leq c_k\log\xi$.
\end{abstract}
\section*{Acknowledgements}
The author thanks Alejandro Rivera for his many pieces of advice regarding probabilities, and Jens Wittsten, Masato Tsujii, Gabriel Rivi\`ere and S\'ebastien Gou\"ezel for interesting discussions about this work, as well as the anonymous referees for their careful reading and constructive suggestions.
\newpage
\tableofcontents

\newpage

\section{Introduction}

This paper focuses on the distribution of the flat traces of iterates
of the transfer operator of a simple example of partially expanding
map.  It is motivated by the Bohigas-Gianonni-Schmidt
\cite{bohigas1984characterization} conjecture in quantum chaos (see
below).\newline In chaotic dynamics, the transfer operator is an
object of first importance linked to the asymptotics of the
correlations.  The collection of poles of its resolvent, called
Ruelle-Pollicott spectrum, can be defined as the spectrum of the
transfer operator in appropriate Banach spaces (see
\cite{ruelle1976zeta} for analytic expanding maps,
\cite{kitaev1999fredholm}, \cite{blank2002ruelle},
\cite{baladi2007anisotropic}, \cite{baladi2008dynamical},
\cite{gouezel2006banach}, \cite{faure2008semi} for the construction of
the spaces for Anosov diffeomorphisms.)\newline

 The study of the Ruelle spectrum for Anosov flows is more difficult
because of the flow direction that is neither contracting nor
expanding.  Dolgopyat has shown in particular in
\cite{dolgopyat1998decay} the exponential decay of correlations for
the geodesic flow on negatively curved surfaces, and Liverani
\cite{liverani2004contact} generalized this result to all $\mc C^4$
contact Anosov flows.  His method involved the construction of
anisotropic Banach spaces in which the generating vector field has a spectral gap, and no longer relies on symbolic dynamics that prevented
from using advantage of the smoothness of the flow.  Tsujii
\cite{tsujii2010quasi} constructed appropriate Hilbert spaces for the
transfer operator of $\mc C^r$ contact Anosov flows, $r\geq 3$ and
gave explicit upper bounds for the essential spectral radii in terms
of $r$ and the expansion constants of the flow.  Butterley and Liverani
\cite{butterley2007smooth} and later Faure and Sjöstrand
\cite{faure2011upper} constructed good spaces for Anosov flows,
without the contact hypothesis.  Weich and Bonthonneau defined in
\cite{bonthonneau2017ruelle} Ruelle spectrum for geodesic flow on
negatively curved manifolds with a finite number of cusps.  Dyatlov and
Guillarmou \cite{dyatlov2016pollicott} handled the case of open
hyperbolic systems.  A simple example of Anosov flow is the suspension
of an Anosov diffeomorphism, or the suspension semi-flow of an
expanding map.  Pollicott showed exponential decay of correlations in
this setting under a weak condition in \cite{pollicott1985rate} and
Tsujii constructed suitable spaces for the transfer operator and gave
an upper bound on its essential spectral radius in
\cite{tsujii2008decay}.\newline

In this article we study a closely related discrete time model, the
skew product of an expanding map of the circle.  It is a particular
case of compact group extension \cite{dolgopyat2002mixing}, which are
partially hyperbolic maps, with compact leaves in the neutral
direction that are isometric to each other.  Dolgopyat showed in
\cite{dolgopyat2002mixing} that the correlation decrease generically
rapidly for compact group extensions, and exponentially in the
particular case of expanding maps.  In our setting of skew-product of
an expanding map of the circle, Faure \cite{faure2011semiclassical}
has shown using semi-classical methods an upper bound on the essential
spectral radius of the transfer operator under a condition shown to be
generic by Nakano Tsujii and Wittsten \cite{nakano2016partial}.  De
Simoi, Liverani, Poquet and Volk \cite{de2017fast} and de Simoi and
Liverani \cite{de2016statistical} \cite{de2018limit} studied fast-slow
dynamical systems, that generalize $\ma T$-extensions of circle
expanding maps.  The roof function, depending on two variables is
multiplied by a small amplitude, and the authors obtained results
about the statistical properties, for long time and small $\varepsilon$.  Arnoldi, Faure, and Weich
\cite{arnoldi2017asymptotic} and Faure and Weich
\cite{faure2017global} studied the case of some open partially
expanding maps, iteration function schemes, for which they found an
explicit bound on the essential spectral radius of the transfer
operator in a suitable space, and obtained a Weyl law (upper bound on
the number of Ruelle resonances outside the essential spectral
radius).  Naud \cite{naud2016rate} studied a model close to the one
presented in this paper, in the analytic setting, in which the
transfer operator is trace-class, and used the trace formula, in the
deterministic and random case to obtain a lower bound on the spectral
radius of the transfer operator.  In the more general framework of
random dynamical systems in which the transfer operator changes
randomly at each iteration, for the skew product of an expanding map
of the circle, Nakano and Wittsten \cite{nakano2015spectra} showed
exponential decay of correlations.\newline

Semiclassical analysis describes the link between quantum dynamics and
the associated classical dynamics in a symplectic manifold.  The
transfer operator happens to be a Fourier integral operator and the
semi-classical approach has thus shown to be useful.  The famous
Bohigas-Giannoni-Schmidt \cite{bohigas1984characterization} conjecture
of quantum chaos states that for quantum systems whose associated
classical dynamic is chaotic, the spectrum of the Hamiltonian shows
the same statistics as that of a random matrix (GUE, GOE or GSE
according to the symmetries of the system)(see also
\cite{gutzwiller2013chaos} and \cite{giannoni1991chaos}).  We are
interested analogously in investigating the possible links between the
Ruelle-Pollicott spectrum and the spectrum of random
matrices/operators.  At first we try to get informations about the
spectrum using a trace formula.  More useful results could follow from
the use of a global normal form as obtained by Faure-Weich in
\cite{faure2017global}.
\subsection{Expanding map}

Let us consider a smooth orientation preserving expanding map $E: \ma T \rightarrow \ma T$ on  the circle $\ma T = \ma R/\ma Z$, that is, satisfying $E'>1$, of degree $l$, and let us call \[m:=\inf E'>1\] and \[M:=\sup E'.\]

\subsection{Transfer operator}\label{1.2}
Let us fix a function $\tau \in C^k(\ma T)$ for some $k\geq 0$.  We are interested in the partially expanding dynamical system on $\ma T \times \ma R$ defined by
\begin{equation}
\label{def_FF}
F(x,y)=\left(E(x),y+\tau(x)\right)\
\end{equation}

We introduce the transfer operator
\[\fonction{\mc L_{\tau}}{\mc C^k(\ma T \times \ma R)}{\mc C^k(\ma T \times \ma R)}u{u\circ F}.\]
\subsection{Reduction of the transfer operator}
Due to the particular form of the map $F$, the Fourier modes in $y$
are invariant under $\mc L_\tau$: if for some $\xi\in\ma R$ and some
$v\in \ma C^k(\ma T)$,
\[u(x,y) = v(x)e^{i \xi y},\]
then
\[\mc L_{\tau} u(x,y)= e^{i\xi\tau(x)} v(E(x)) e^{i\xi y}.\]

Given $\xi \geq 0$ and a function $\tau$, let us consequently consider the transfer operator  $\mc L_{\xi,\tau}$ defined on functions $v \in C^k( \ma T)$ by
\[\forall x\in \ma T,\mc L_{\xi,\tau}  v(x):=e^{i\xi\tau(x)}v(E(x)),\]

\subsection{Spectrum and flat trace}

In appropriate spaces, the transfer operator has a discrete spectrum outside a small disk, the eigenvalues are called Ruelle resonances.
It is in general not trace-class, but one can define its flat trace (see Appendix \ref{appendice Ruelle spectrum} for a more precise discussion about Ruelle resonances, flat trace and their relationship).
\begin{lemma}[Trace formula, \cite{atiyah1967lefschetz}, \cite{guillemin1977lectures}]
For any $\mc C^0$ function $\tau$ on $\ma T$, the flat trace of $\mc L^n_{\xi,\tau}$ is well defined and
\begin{equation}\label{trace}\mathrm{Tr}^\flat\mc L^n_{\xi,\tau}= \sum_{x,E^n(x)=x}\frac{e^{i\xi\tau_x^n}}{{(E^n)'(x)}-1},
\end{equation}
where $\tau^n_x$ denotes the Birkhoff sum: For a function $\phi \in C(\ma T)$  and a point $x\in \ma T$   we define
\begin{equation}\label{birkhoff}\phi_x^n := \sum\limits_{k=0}^{n-1}\phi(E^k(x)).\end{equation}

\end{lemma}
\subsection{Gaussian random fields}

We define our random functions on the circle by means of their Fourier coefficients.  We are only interested in $\mc C^0$ functions.  We will  denote by $\mc N(0,\sigma^2)$ (respectively  $\mc N_{\ma C}(0,\sigma^2)$) the real (respectively complex) centered Gaussian law of variance $\sigma^2$, with respective densities \[ \frac{1}{\sigma\sqrt{2\pi}} e^{-\frac{1}{2\sigma^2}x^2} \text{ and } \frac{1}{\sigma\pi} e^{-\frac{1}{\sigma^2}|z|^2}.\] With these conventions, a random variable of law $\mc N_{\ma C} (0,\sigma^2)$ has independent real and imaginary parts of law $\mc N(0,\frac{\sigma^2} 2)$, and the variance of its modulus is consequently $\sigma^2$.
\begin{defi}
We will call centered stationary Gaussian random fields on $\ma T$ the real random distributions $\tau$ whose Fourier coefficients $\left(c_p(\tau)\right)_{p\geq 1}$ are independent complex centered Gaussian random variables, with variances growing at most polynomially, such that $c_0(\tau)$ is a real centered Gaussian variable independent of the $c_p(\tau),\ p\geq 1$.  The negative coefficients are necessarily given by \[c_{-p}(\tau)=\overline{c_p(\tau)}.\]
\end{defi}
The Gaussian fields are in general defined as distributions if their Fourier coefficients have variances with polynomial growth  and the decay of the variances of the coefficients gives sufficient conditions for the regularity of the field.
\begin{lemma}\label{regularite}
If  $\ma E[|c_p(\tau)|^2]$ has a polynomial growth, $\tau=\sum_p c_p(\tau) e^{2i\pi p\cdot}$ defines almost surely a distribution: almost surely
\[\forall \phi= \sum c_p(\phi) e^{2i\pi p\cdot}\in\mc C^\infty(\ma T),\ \langle\tau,\phi\rangle:=\sum_p \overline{c_p(\tau)}c_p(\phi)<\infty.\]
Let $k\in\ma N$.If for some $\eta>0$ \begin{equation}\label{ck}
    \ma E\left[\abs{c_p(\tau)}^2\right]= O\left(\frac 1{p^{2k+2+\eta}}\right).
\end{equation}
Then $\tau$ is almost surely $\mc C^k$.
\end{lemma}
\begin{proof}
See appendix \ref{Borel-Cantelli}.
\end{proof}
In what follows we will always assume that (\ref{ck}) is satisfied, at least for $k=0$, so that our random fields are random variables on $\mc C^0(\ma T)$.  This will ensure the existence of flat traces.

\subsection{Result}

If $x$ is a periodic point, let us write its prime period
\[l_x:=\min\{k\geq1,E^k(x)=x\}.\]

Let us define for every $n\in\ma N$:
\begin{equation}
    \label{def_An}
A_n:=\left(\sum\limits_{E^n(x)=x}\frac {l_x}{((E^n)'(x)-1)^2}\right)^{-\frac 1 2}
\end{equation}


\begin{theorem}\label{main}
Let $k\in\ma N$.  Let $\tau_0\in\mc C^k(\ma T)$.
Let \[\delta\tau=\sum_{p\in\ma Z}c_p e^{2i\pi p\cdot}\]
be a centered Gaussian random field, such that $\ma E[|c_p|^2]=O(p^{-2-\nu})$ for some $\nu>0$.  This way, $\delta\tau$ is a.s. $\mc C^0$.
If
\begin{equation}\label{condition} \exists\epsilon>0, \exists C>0, \forall p\in\ma Z^*, \ma E\left[\left|c_p\right|^2\right]\geq \frac C{p^{2k+2+\epsilon}},\end{equation}

then one has the convergence in law of the flat traces

\begin{equation}\label{theorem}
    A_n\mathrm{Tr}^\flat\left(\mc L^n_{\xi, \tau_0+\delta\tau}\right)\longrightarrow\mc N_{\ma C}(0,1)
\end{equation}
as $n$ and $\xi$ go to infinity, under the constraint 
\begin{equation}\label{cond2}
\exists 0<c<1, \forall  n,\xi,\  n\leq c\frac{\log\xi}{\log l+(k+\frac 12+\frac\epsilon2)\log M}.
\end{equation}
Note that condition (\ref{condition}) can allow $\tau$ to be $\mc C^k$ by Lemma \ref{regularite}.
\end{theorem}
\begin{remark}
The statement implies that the convergence still holds if we multiply $\delta\tau$ by an arbitrarily small number $\eta>0$.  For instance for $\tau_0 = 0$, \[A_n\mathrm{Tr}^\flat\left(\mc L^n_{\xi, 0}\right)\longrightarrow\infty\] at exponential speed, uniformly in $\xi$, but if $\delta\tau$ is an irregular enough Gaussian field in the sense of (\ref{condition}), then for any $\eta>0$ and $c<1$ holds
\begin{equation*}
    A_n\mathrm{Tr}^\flat\left(\mc L^n_{\xi, \eta \cdot\delta\tau}\right)\longrightarrow\mc N_{\ma C}(0,1)
\end{equation*}
under condition (\ref{cond2}).
\end{remark}
\begin{remark}
 Condition (\ref{cond2}) means that time $n$ is smaller than a constant times  the Ehrenfest time $\log \xi$, and this constant decreases with the regularity $k$ of the field $\delta \tau$.
\end{remark}

\subsection{Sketch of proof}

The proof is based around the following arguments:
\begin{enumerate}
    \item Note first that the convergence (\ref{theorem}) is satisfied if all the phases appearing in (\ref{trace}) are independent and uniformly distributed. \begin{remark}
    For sake of simplicity, in this sketch of proof, we will state pairwise independence for the phases in (\ref{trace}), while in fact we must pack them by orbits, since Birkhoff sums $\phi^n_x$ are the same on all the orbit, but this changes little to the problem.  For instance this simplification would remove the factor $l_x$ in the definition (\ref{An}) of $A_n$ corresponding to this multiplicity.
    \end{remark}The convergence can be deduced from the standard proof of the central limit theorem showing pointwise convergence of the characteristic function.  However, here, since the periodic points are dense in $\ma T$, requiring independence of the values $(\delta\tau(x))_{E^n(x)=x}$ would lead to very bad regularity of the field (it is not hard to see that it would be almost surely nowhere locally bounded). 
    \item We fix a Gaussian field $\delta\tau=\sum c_pe^{2i\pi p\cdot}$ fulfilling the hypothesis of Theorem \ref{main} and start by constructing an auxiliary field with the same law and show that it satisfies the convergence (\ref{theorem}).  This is sufficient since the convergence in law only involves the law of the random field.
    
    \item For each $j\geq 1$, we construct a smooth random field $\delta\tau_j$, such that for any pair of periodic points $x\neq y$ of period $j$, $\delta\tau_j(x)$ and $\delta\tau_j(y)$ are independent. 
    Since by (\ref{trace}) $\mathrm{Tr}^\flat(\mc L^n_{\xi,\tau})$ only involves points of period $n$, the phases appearing at time $n$, for the function $\delta\tau_n$, in $\mc L^n_{\xi,\delta\tau_n}$ are consequently all independent random variables on $S^1.$
    If moreover $\xi$ is large enough, the variables $\xi\left(\delta\tau_n\right)^n_x$ are Gaussian with large variances, so $\xi\left(\delta\tau_n\right)^n_x\mod 2\pi$ (and therefore the phases $e^{i\xi\left(\delta\tau_n\right)^n_x}$) are close to be uniform.  Thus, the convergence (\ref{theorem}) should hold for $\mathrm{Tr}^\flat(\mc L^n_{\xi,\delta\tau_n})$ under a certain relation between $n$ and $\xi$ that will be explained in number (8). 
    
    \item An important point is that if the phases $(e^{i\xi(\delta\tau_n)^n_x})_{\{x\in\ma T,E^n(x)=x\}}$ are independent and close to be uniform, then adding to $\delta\tau_n$ an independent field will not change this fact, as the following lemma suggests:
 \begin{lemma}\label{indep}
Let $X,X'$  be real independent random variables such that $e^{iX}, e^{iX'}$ are uniform on $S^1$.  Let $Y,Y'$  be real  random variables such that  $X$ and $ X'$ are  independent of both $Y$ and $Y'$.  Then $e^{i(X+Y)}$ and $e^{i(X'+Y')}$ are still independent uniform random variables on $S^1$.
\end{lemma}
Note that no independence between $Y$ and $Y'$ is needed.
See appendix \ref{annexe3} for the proof.

\item Using this analogy, if the fields $\delta\tau_j$ are chosen independent, it should follow that the convergence (\ref{theorem}) holds for  $\mathrm{Tr}^\flat\left(\mc L^n_{\xi,\sum_{j\geq 1}\delta\tau_j}\right)$ for large $\xi$.
\item  The fields $\delta\tau_j$ are almost surely smooth.  However, because the distance between periodic points decreases as $M^{-j}$ according to Lemma \ref{periodic}, if we want to be sure that $\sum_j\delta\tau_j$ is $\mc C^k$, and $\ma E[\delta\tau_j(x)\delta\tau_j(y)]=0$ for all $x\neq y$ of period $j$, let us see that we need to impose an exponential decay of the standard deviation (independent of the point $x$):
\begin{equation}\label{heuristique ecart type}
    \sqrt{\ma E[|\delta\tau_j(x)|^2]}\approx M^{-j(k+\frac 12+\varepsilon)}
\end{equation}  for some $\varepsilon>0.$ This can be deduced heuristically from the fact (see Definition \ref{defi_covariance} below) that \begin{equation}\label{covariance sketch}
    \ma E[\delta\tau_j(x)\delta\tau_j(y)]=\sum_p \ma E[|c_p(\delta\tau_j)|^2] e^{ip(x-y)}=:K_j(x-y)
\end{equation}and the uncertainty principle: a localisation of $K_j$ at a scale $M^{-j}$ implies non negligible coefficients $\ma E[|c_p(\delta\tau_j)|^2]$ for $p$ of order $M^j$.
Let us for instance assume that the Fourier coefficients $\ma E[|c_p(\delta\tau_j)|^2]$ of $K_j$ write
\begin{equation}
    \ma E[|c_p(\delta\tau_j)|^2]=\alpha_j^2f\left(\frac p{M^j}\right)^2
\end{equation} for some amplitudes $\alpha_j$ to determine and some positive Schwartz function $f:\ma R\longrightarrow\ma R.$ Then,
since
\[\delta\tau_j=\sum_p \sqrt{\ma E[|c_p(\delta\tau_j)|^2]} \zeta_{j,p}e^{2i\pi p\cdot}\] for i.i.d. $\mc N(0,1)$ random variables $\zeta_{j,p}$, roughly,
\[\begin{split}
    \sup|\delta\tau_j^{(k)}|&\approx\alpha_j\sum_p |p|^k f\left(\frac p{M^j}\right)\\
                            &=\alpha_j M^{j(k+1)}\frac1{M^j}\sum_p\frac{|p|^k}{M^{jk}}f\left(\frac p{M^j}\right)\\
                            &\sim C \alpha_j M^{j(k+1)}.
\end{split}\]
(The second line involved a Riemann sum.)
Consequently, with those approximations, choosing $\alpha_j = M^{-j(k+1+\varepsilon)}$ gives a $\mc C^k$ function $\sum_{j\geq 1}\delta\tau_j$.  Then,
\[\begin{split}
    \ma E[|\delta\tau_j(x)|^2]&\underset{(\ref{covariance sketch})}=\sum_p\ma E[|c_p(\delta\tau_j)|^2]\\
                              &\underset{\phantom{(\ref{covariance})}}= \sum_p\alpha_j^2f\left(\frac p{M^j} \right)^2\\
                              &\underset{\phantom{(\ref{covariance})}}=\alpha_j^2M^j\frac1{M^j}\sum_pf\left(\frac p{M^j} \right)^2\\
                              &\underset{\phantom{(\ref{covariance})}}\sim C \alpha_j^2M^j = M^{-j(2k+1+2\varepsilon)}
\end{split}\]
as announced.
\item This condition, together with (\ref{condition}) can easily be shown to imply that the Fourier coefficients $\tilde c_p$ of $\sum_{j\geq1}\delta\tau_j$ satisfy
\[\ma E[|\tilde c_p|^2]\leq C\ma E[|c_p|^2].\]
This allows us to define a field $\delta\tau_0$, that we chose independent from the other $\delta\tau_j$, by \[\ma E[|c_p(\delta\tau_0)|^2]= C\ma E[|c_p|^2]-\ma E[|\tilde c_p|^2],\] so that $\frac1C\sum_{j\geq0}\delta\tau_j$ has the same law as $\delta\tau$ and still satisfies the convergence (\ref{theorem}) for $\xi$ large enough from (4) of this sketch.
\item To get an idea of the origin of the relation (\ref{cond2}) between $n$ and $\xi$, let us assume that we want all the arguments $\xi(\delta\tau_n)^n_x$ in $\mathrm{Tr}^\flat(\mc L^n_{\xi,\delta\tau_n})$ to go uniformly to infinity in order to get approximate uniformity of the phases and thus convergence towards a Gaussian law.  Note that for any $x$, \begin{equation}\label{tac}
    \ma P\left[\frac{\left|(\delta\tau_n)^n_x\right|}{ \sqrt{\ma E[|(\delta\tau_n)^n_x|^2]}}\leq\epsilon\right]\underset{\epsilon\to0}=O(\epsilon).
\end{equation}Let $(C_n)$ be a sequence going to infinity.

(\ref{tac}) implies  \[\begin{split}
    \ma P\left[\bigcap_{E^n(x)=x}\left\{\xi(\delta\tau_n)^n_x> C_n\right\}\right]&\ \ \ \ =1-\ma P\left[\exists x,E^n(x)=x,\xi(\delta\tau_n)^n_x\leq C_n \right]\\
                                                                  &\ \ \ \ \geq 1- \sum_{E^n(x)=x}\ma P\left[\xi(\delta\tau_n)^n_x\leq C_n \right]\\
                                                                  &\underset{\mathrm{Lemma } \ref{periodic}}= 1-(l^n-1) \ma P\left[\xi(\delta\tau_n)^n_x\leq C_n \right]\\
                                                                  &\ \ \ \underset{(\ref{tac})}\geq 1-Cl^n \frac{C_n}{\xi\sqrt{\ma E[|(\delta\tau_n)^n_x|^2]}}
\end{split}\]
if $x$ denotes any point and $\xi\gg\frac{C_n}{\sqrt{\ma E[|(\delta\tau_n)^n_x|^2]}}.$ By independence \[\begin{split}
    \sqrt{\ma E[|(\delta\tau_n)^n_x|^2]}&\underset{\phantom{(\ref{heuristique ecart type})}}=\left(\sum_{k=0}^{n-1} {\ma E[|\delta\tau_n(E^k(x))|^2]}\right)^{\frac12}\\
                                        &\underset{(\ref{heuristique ecart type})}\approx \sqrt n M^{n(k+\frac12+\varepsilon)}
\end{split}\]
for some $\varepsilon>0.$
Thus
\[\ma P\left[\xi(\delta\tau_n)^n_x\to\infty \mathrm{\ uniformly}\ \mathrm{w.r.t.\ }x\mathrm{\ s.t.\ }E^n(x)=x\right]\to1\]
for $\xi\gg l^n M^{n(k+\frac 12+\varepsilon)}, $ which gives (\ref{cond2}).
     \end{enumerate}

 \section{Numerical experiments}
 We consider an example with the non linear expanding map
 \begin{equation}
 \label{ex_E}
 E(x) = 2x + 0.9/(2\pi)  \sin (2\pi (x+0.4) )    
 \end{equation}
 plotted on Figure  \ref{fig:dessin_E}.
 In Figure \ref{fig:dessin2}, we have the histogram of 
 the modulus $ S = \left| A_n\mathrm{Tr}^\flat\left(\mc L^n_{\xi, \tau_0+\delta\tau}\right) \right| $ obtained after a sample of $10^4$ random functions $\delta \tau$.
 We compare the histogram with the function $  C S \exp (-S^2) $ in red, i.e.  the radial distribution of a Gaussian function, obtained from the prediction of Theorem \ref{main}.
 We took $n=11$, $\xi = 2 .10^6$, $\tau_0 = \cos(2 \pi x)$.  We also observe a good agreement for the (uniform) distribution of the arguments that is not represented here.

  \begin{figure}[h!]  \includegraphics[scale=0.3]{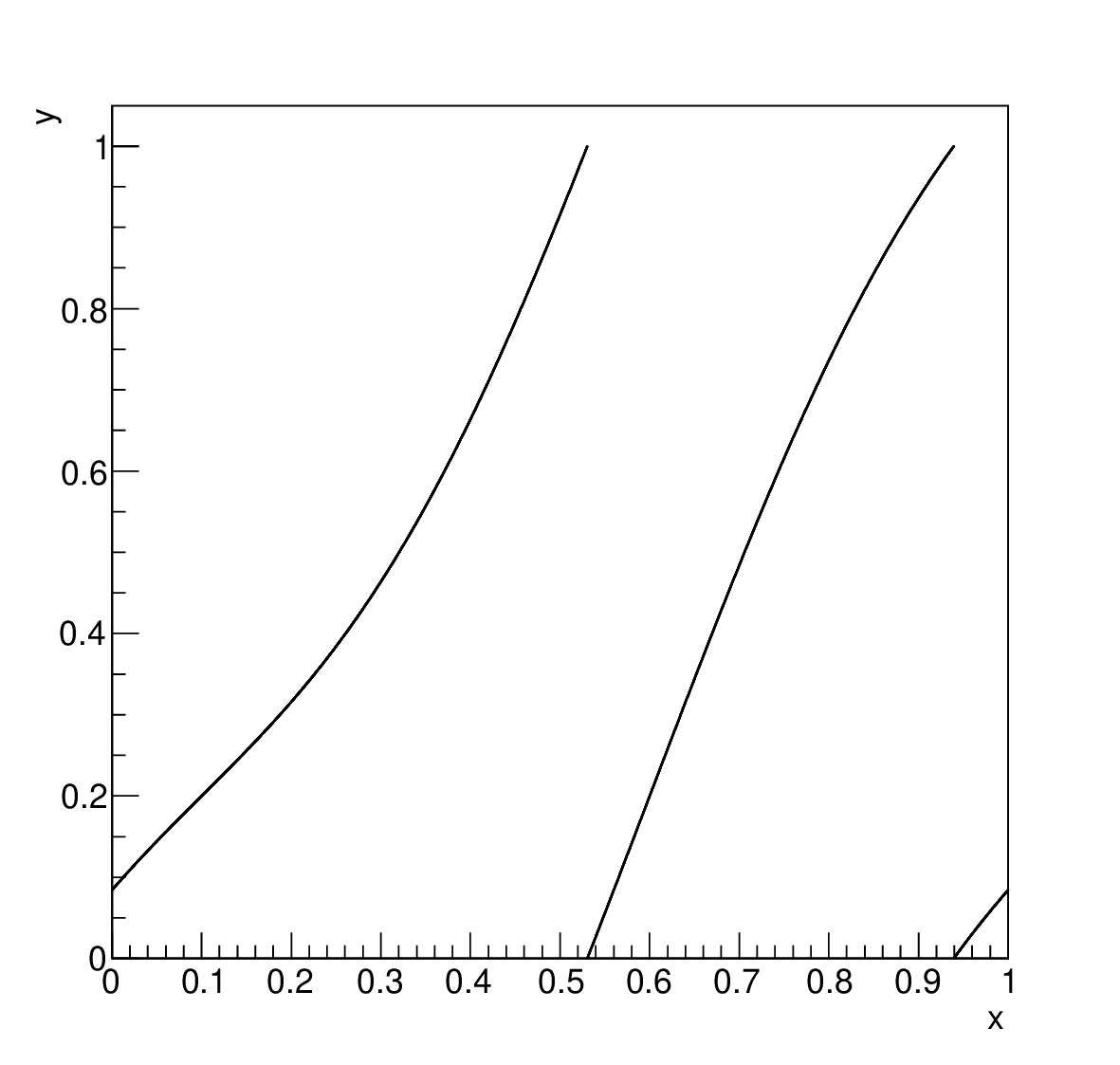}
  \caption{Graph of the expanding map $E(x)$ in Eq.(\ref{ex_E})} \label{fig:dessin_E} \end{figure}

  \begin{figure}[h!]  \includegraphics[scale=0.4]{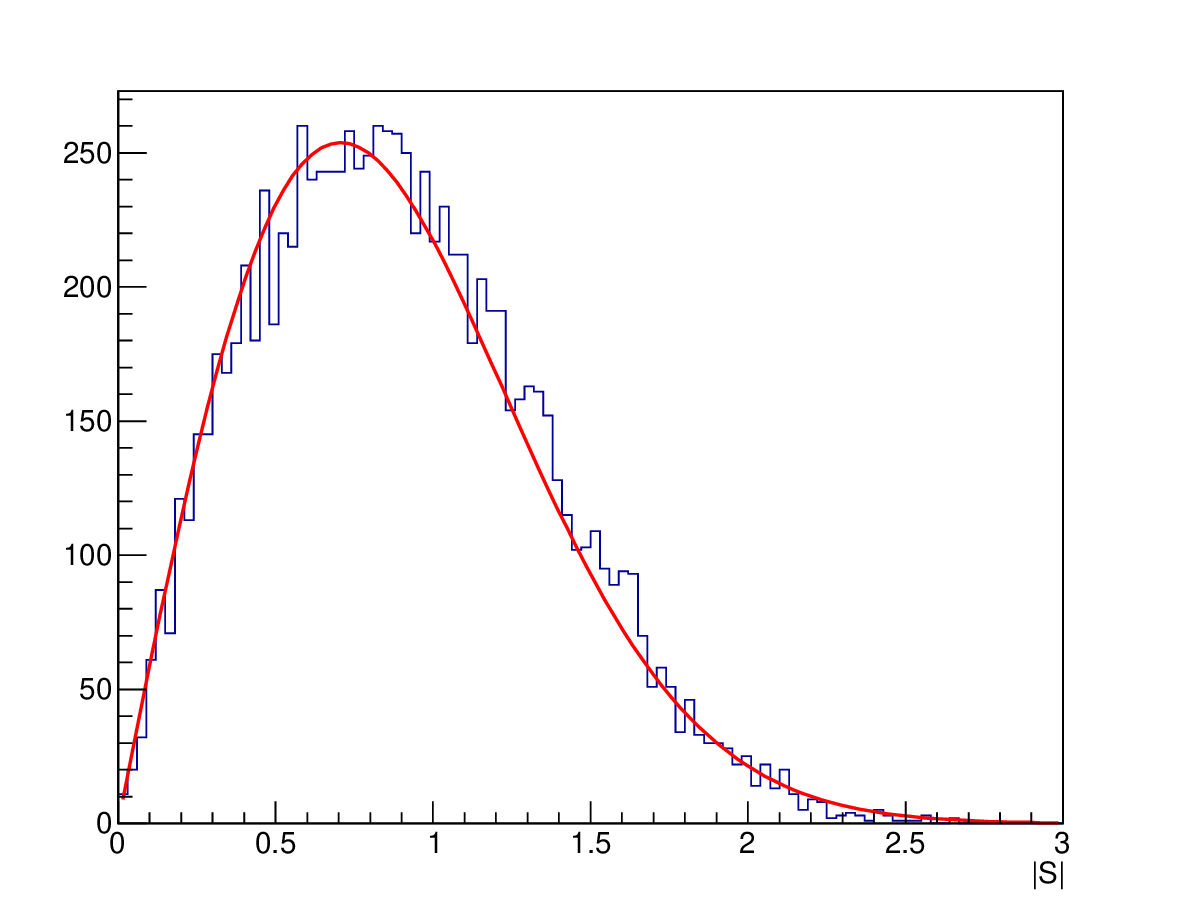}
  \caption{In blue, the histogram of $ S = \left| A_n\mathrm{Tr}^\flat\left(\mc L^n_{\xi, \tau_0+\delta\tau}\right) \right| $ for $n=11$, $\xi = 2 .10^6$, $\tau_0 = \cos(2 \pi x)$ and  the sample $10^4$ random functions $\delta \tau$.  The histogram is  well fitted by $  C S \exp (-S^2) $ in red, as predicted by Theorem  \ref{main}} \label{fig:dessin2} \end{figure}
 
 \newpage
\section{Proof of theorem \ref{main}}
A stationary centered Gaussian random field is characterized by its covariance function:
\begin{defi}\label{defi_covariance}
Let $\tau=\sum_{p\in\ma Z}c_pe^{2i\pi p\cdot}$ be a stationary centered Gaussian random field, satisfying \[\ma E[|c_p|^2]=O\left(\frac{1}{p^{2+\eta}}\right)\] for some $\eta>0,$ so that $\tau$ is almost surely $\mc C^0$ according to Lemma \ref{regularite}.
Let us define its covariance function $K$ by
\begin{equation}\label{covariance}
    K(x):=\sum_p\ma E[|c_p|^2]e^{2i\pi p x}.
\end{equation}
For any pair of points $(x,y)\in\ma T^2,$ we have
\begin{equation}\label{equation covariance}
    \ma E\left[{\tau(x)}\tau(y)\right]=K(x-y).
\end{equation}
\end{defi}
\begin{proof}[Proof of the last statement]
Remark from Appendix \ref{Borel-Cantelli} that the condition $\ma E[|c_p|^2]=O\left(\frac{1}{p^{2+\eta}}\right)$ implies that $\tau$ is almost surely equal to its Fourier series.  Thus,
\[\begin{split}
    \ma E\left[{\tau(x)}\tau(y)\right]&=\sum_{p,q\in\ma Z}\ma E[c_p(\tau)c_q(\tau)]e^{2i\pi(px+qy)}\\
    &=\sum_{p\in\ma Z}\left(\ma E[\abs{c_p}^2]e^{2i\pi p(x-y)}+\ma E[{c_p}^2]e^{2i\pi p(x+y)}\right)
\end{split}\]
from the independence relationships of the Fourier coefficients.
Now, 
\[\ma E[{c_p}^2] = \mathbb E [(\mathrm{Re}(c_p))^2]-\ma E[(\mathrm{Im}(c_p))^2]+2i\ma E[(\mathrm{Re}(c_p))(\mathrm{Im}(c_p))] =0.\]
\end{proof}
\subsection{Definition of a Gaussian field satisfying Theorem \ref{main}}
Let us fix a random centered Gaussian field $\delta\tau=\sum_{p\in\ma Z}c_pe^{2i\pi p\cdot}$ satisfying the hypothesis of Theorem \ref{main}.
 Let us define the Gaussian fields mentioned in step (3) of the sketch of proof.  Let $K_{\text{init}}\in\mc C^\infty_c(\ma R)$ be a smooth function supported in $\left[-\frac 13,\frac13\right]$, with non negative Fourier transform, satisfying \footnote{To construct such a function, take a non zero even function $g\in\mc C^\infty_c(\ma R)$. $g$ has a real Fourier transform.  Then $g*g\in\mc C^\infty_c(\ma R)$ and its Fourier transform is $\hat g^2\geq 0$.  Moreover $g*g(0)=\int \hat g^2>0$. }
 \begin{equation}
 \label{K(0)}
     K_{\text{init}}(0)=1.
 \end{equation}
 
 Let $k\geq0$ be the integer involved in Theorem \ref{main} giving the regularity of the field.  Let $\epsilon>0$ be the constant appearing in Theorem \ref{main} and define for any integer $j\geq 1$
\begin{equation}
    \label{kj}K_j(x) = \frac 1 {M^{j(2k+1+\epsilon)}} K_{\text{init}}(M^jx).
\end{equation}

The Fourier transform of $K_j$ is given by 
\begin{equation}\label{Fourier}\widehat{K_j}(\xi) = \frac 1 {M^{j(2k+2+\epsilon)}} \widehat{K_{\text{init}}}\left(\frac\xi{M^j}\right)\geq0\end{equation}
 The functions  $K_j$, for all $j\geq 1$, are  supported in $\left[-\frac 13,\frac13\right]$ and can then be seen as functions on the circle $\ma T$ by trivially periodizing them.  Let $c_{p,j}$, for $p\geq 0, j\geq 1$ be independent centered Gaussian random variables of respective variances $\hat K_j(2\pi p)$, and let us write
\[\delta\tau_j=\sum_p c_{p,j}e^{2i\pi p\cdot},\]
where $c_{-p,j}:=\overline{c_{p,j}},p\geq 1$.  Note that, since $K_j$ is smooth for all $j,$ the variances $\hat K_j(2\pi p)$ of $c_{p,j}$ decay rapidly with $p$ (for fixed $j$), and therefore, each $\delta\tau_j$ is almost surely smooth by Lemma \ref{lemme Alejandro}.
 \begin{lemma}
$\sum_{j\geq1}\delta\tau_j$ is a centered Gaussian random field $\sum \tilde c_p e^{2i\pi p\cdot}$ and \[\ma E\left[|\tilde c_p|^2\right]=O\left(\ma E\left[|c_p|^2\right]\right).\]
 \end{lemma}
 \begin{proof}
 We have seen in Eq.(\ref{Fourier}) that
 \[\widehat{K_j}(\xi) = \frac 1 {M^{j(2k+2+\epsilon)}} \widehat{K_{\text{init}}}(\frac\xi{M^j}).\]
 Since $K_{\text{init}}$ is smooth, there exists a constant $C>0$ such that
 \[\forall\xi\in\ma R, \widehat{K_{\text{init}}}(\xi)\leq \frac C{\left\langle\xi\right\rangle^{2k+2+\frac\epsilon 2}},\]
 with the usual notation $\left\langle\xi\right\rangle=\sqrt{1+\xi^2}\geq\abs\xi$.
 Thus,
\[\begin{split}
    \ma E\left[|c_{p,j}|^2\right]&=\frac{1}{M^{j(2k+2+\epsilon)}}\widehat{K_{\text{init}}}(\frac{2\pi p}{M^j})\\
    &\leq \frac{C}{M^{j\frac\epsilon 2}}\frac1{\abs{2\pi p}^{2k+2+\frac\epsilon2}}.
\end{split}\]
Consequently, since by independence
\[\ma E\left[\left|\tilde c_p\right|^2\right] = \ma E\left[\left|\sum_{j\geq1}c_{p,j}\right|^2\right]=\sum_{j\geq1}\ma E\left[\left|c_{p,j}\right|^2\right], \]
\[\ma E\left[\left|\tilde c_p\right|^2\right] =O\left(\frac{1}{\abs{2\pi p}^{2k+2+\frac\epsilon 2}}\right)\underset{(\ref{condition})}=O(\ma E\left[|c_p|^2\right]).\]

\end{proof}
Thus, fixing a constant $C$ such that
\[C\ma E\left[|c_p|^2\right]\geq \ma E\left[\left|\tilde c_p\right|^2\right] ,\]
we can define a random Gaussian field $\delta\tau_0=\sum_{p\in\ma Z}c_{p,0}e^{2i\pi p\cdot}$ with coefficients $c_{p,0}$ independent from the $c_{p,j}$ such that
\[\ma E\left[|c_{p,0}|^2\right]=C\ma E\left[|c_p|^2\right]-\ma E\left[\left|\tilde c_p\right|^2\right].\]
This way $\frac 1 C \sum_{j\geq0}\delta\tau_j$ and $\delta\tau$ have the same law.  By this we mean that their Fourier coefficients have the same laws.  By our hypothesis, the convergence of the Fourier series are almost surely normal, thus for any finite subset $\{x_k\}_k$ of $\ma T$, $(\frac 1 C\sum\delta\tau_j(x_k))_k$ and $(\delta\tau(x_k))_k$ have the same law.  Therefore, the laws of $\mathrm Tr^\flat(\mc L^n_{\xi,\tau_0+\delta\tau})$ and $\mathrm Tr^\flat(\mc L^n_{\xi,\tau_0+\frac1 C\sum\delta\tau_j})$ are the same, and the convergence of Theorem \ref{main} is equivalent to 
\begin{equation}
    \label{convergence bis}
    A_n\mathrm Tr^\flat(\mc L^n_{\xi,\tau_0+\sum\delta\tau_j})\overset{\mc L}\longrightarrow\mc N_{\ma C}(0,1)
\end{equation}
under condition (\ref{condition}). (The constant $\frac1C$ can be 'absorbed' in $\xi$ up to the replacement of $\tau_0$ by $C\tau_0$ that has no consequence.)
In the rest of the paper we will show (\ref{convergence bis}) and will write \begin{equation}
\label{defi tau}
    \tau:=\tau_0+\sum_{j\geq 0}\delta\tau_j.
\end{equation}
\subsection{New expression for $\mathrm{Tr}^\flat(\mc L^n_{\xi,\tau})$}
 
We will write the set of periodic orbits of (non primitive)  period $n$  as
\begin{equation}\label{Pern}
    \mathrm{Per}(n):=\left\{\{x,E(x),\cdots,E^{n-1}(x)\}, E^n(x)=x, x \in \ma T\right\},
\end{equation}
and the set of periodic orbits of primitive  period  $n$ as
\begin{equation}\label{Pm}
    \mathcal{P}_n:=\left\{\{x,E(x),\cdots,E^{n-1}(x)\}, n =\min\{k\in\ma N^*,E^k(x)=x\} , x \in \ma T \right\}.
\end{equation}
This way, $\mathrm{Per}(n)$ is the disjoint union
\begin{equation}\label{prime}
    \mathrm{Per}(n)=\coprod_{m|n}\mc P_m.
\end{equation}
Let us rewrite the sum $\mathrm{Tr}^\flat(\mc L^n_{\xi,\tau})$, where $\tau $ is given by (\ref{defi tau}).  We know from (\ref{trace}) that
\[\begin{split}
    \mathrm{Tr}^\flat(\mc L^n_{\xi,\tau})&= \sum\limits_{E^n(x)=x}\frac{e^{i\xi\tau_x^n}}{{(E^n)'(x)}-1}\\
                                         &=\sum\limits_{E^n(x)=x}\frac{e^{i\xi\tau_x^n}}{{e^{J^n_x}-1}},
\end{split}\]
where $J(x) = \log (E'(x))>0$ and $J_x^n$ is the Birkhoff sum as defined in (\ref{birkhoff}).
If $f^n_O$ stands for the Birkhoff sum $f^n_x$ for any $x\in O$ , let us write
\begin{equation}
\label{Tr_F}
\mathrm{Tr}^\flat(\mc L^n_{\xi,\tau})= \sum_{m|n}m\sum_{O\in \mc P_m}\frac{e^{i\xi\tau^n_O}}{e^{J^n_O}-1}.    
\end{equation}

For $O\in \mathrm{Per}(n)$, we can write
\[\tau^n_O = (\delta\tau_n)^n_O + \sum_{j\neq n}(\delta\tau_j)^n_O + (\tau_0)^n_O.\]
Since the covariance function $K_n$ is supported in $\left[-\frac 1{3 M^n},\frac1{3M^n}\right]$, we deduce from Lemma \ref{periodic} and (\ref{equation covariance}) that the values taken by $\delta\tau_n$ at different periodic points of period dividing $n$, which have law $\mc N(0, K_n(0))$  are independent random variables.\newline 
Thus, for $n\in\ma N$, $m|n$ and $O\in\mc P_m$, $(\delta\tau_n)^m_O$ is a centered Gaussian random variable of variance $mK_n(0)$, and $(\delta\tau_n)^n_O=\frac nm(\delta\tau_n)^m_O$ has variance $(\frac{n}{m})^2mK_n(0) = \frac{n^2}{m}K_n(0)$.  

\begin{defi}
\label{C}
We say that two families of real random variables $(X_O^n)_{\begin{subarray}{l}n\geq 1  \\O\in \mathrm{Per}(n)\end{subarray} }$ and $(Y_O^n)_{\begin{subarray}{l}n\geq 1  \\O\in \mathrm{Per}(n)\end{subarray} }$   satisfy condition (C) if 
\begin{enumerate}
    \item for every $m|n$, and $O\in\mc P_m$, $X^n_O$ has law $\mc N(0,\frac {n^2} m K_n(0))$,
    \item for every  $O'\neq O\in\mathrm{Per}(n)$ and every $O''\in\mathrm{Per}(n)$,  $X^n_O $ is independent of $X^n_{O'}$ and  $Y^n_{O''}$.
\end{enumerate}
\end{defi}

Writing $X^n_O = (\delta\tau_n)^n_O $ and $Y^n_O = \sum_{j\neq n}(\delta\tau_j)^n_O + (\tau_0)^n_O$, we have obtained 
\begin{lemma}\label{orbit}
There exist families of random variables $(X^n_O), (Y^n_O)$ satisfying condition (C) of Definition (\ref{C}) such that for every $n\geq 1$ and $O\in\mathrm{Per}(n)$\begin{equation}
\label{tau_On}
    \tau^n_O= X^n_O+Y^n_O
\end{equation}
\end{lemma}
In order to adapt the proof of lemma \ref{indep}, we want to show that for large $\xi$, the random variables $e^{i\xi (X^n_O+Y^n_O)}, O\in\mathrm{Per}(n)$ are close to be independent and uniform  on $S^1$.

\begin{remark}
\label{Char_function}
We have
\begin{equation}
\label{Tr_F2}
\mathrm{Tr}^\flat(\mc L^n_{\xi,\tau}) \underset{(\ref{Tr_F}),(\ref{tau_On})}= \sum_{m|n}m\sum_{O\in \mc P_m}\frac{e^{i\xi(X_O^n + Y_O^n)}}{e^{J^n_O}-1}.
\end{equation}

Our aim is to approximate the characteristic function of $A_n\mathrm{Tr}^\flat(\mc L^n_{\xi,\tau})$ which is the expectation of
\begin{multline}
    \label{BE2}
\exp\left(iA_n \left( 
\mu \mathrm{Re}(\mathrm{Tr}^\flat(\mc L^n_{\xi,\tau}))
+ \nu \mathrm{Im}(\mathrm{Tr}^\flat(\mc L^n_{\xi,\tau}))
\right)\right)   
=\\
\prod\limits_{m|n}\prod\limits_{O\in\mc P_m}
\exp\left[i\frac{mA_n}{e^{J_O^n}-1}  \left(
 \mu\cos(\xi(X_O^n + Y_O^n))+ \nu\sin(\xi(X_O^n + Y_O^n)) 
\right)\right]
\end{multline}
    for fixed, $\mu,\nu\in\ma R.$
The right hand side of (\ref{BE2}) can be written as
\[\prod\limits_{m|n}\left(\prod\limits_{O\in\mc P_m} f_O\left(e^{i\xi(X^n_{O}+Y^n_{O})}\right)\right)^m\]
for some continuous functions $f_O:S^1\longrightarrow\ma C$ (depending on $\mu,\nu$):
\[f_O(z)=\exp\left[i\frac{A_n}{e^{J_O^n}-1}\left(\mu\mathrm{Re}(z)+\nu\mathrm{Im}(z)\right)\right]\] In the next Lemma we first consider indicator functions on $S^1$ for $f_O$.
\end{remark}

\begin{lemma}\label{2.5}
 Let $(X_O^n)_{\begin{subarray}{l}n\geq 1  \\O\in Per(n)\end{subarray} }$ and $(Y_O^n)_{\begin{subarray}{l}n\geq 1  \\O\in Per(n)\end{subarray} }$ be two families of real random variables satisfying  satisfying condition (C) of Definition (\ref{C}).  Assume that  $n$ and $\xi$ satisfy (\ref{cond2}).
Then there is a constant $C>0$ such that for every $n\in\ma N$ and every real numbers $(\alpha_O)_{O\in\mathrm{Per}(n)},(\beta_O)_{O\in\mathrm{Per}(n)} $ such that \[\forall O\in\mathrm{Per}(n),\ 0<\beta_O-\alpha_O<2\pi,\] for every complex numbers $(\lambda_O)_{O\in\mathrm{Per}(n)}$,
if $A_O:=e^{i]\alpha_O,\beta_O[}\subset S^1\subset\ma C$   and $\mathds 1_{A_O}:S^1\rightarrow\ma C$ is the characteristic function of  $A_O$, we have
\begin{equation}
   \label{complique} 
\left|\frac{\ma E\left[\prod\limits_{m|n}\left(\prod\limits_{O\in\mc P_m}\lambda_O\mathds 1_{A_O}\left(e^{i\xi(X^n_{O}+Y^n_{O})}\right)\right)^m\right]}{\prod\limits_{m|n}\prod\limits_{O\in\mc P_m}\lambda_O^m \left( \frac{\beta_O-\alpha_O}{2\pi} \right)}-1\right|\leq \xi^{c-1}n^{-\frac12}\left(\underset{(\ref{cond2})}\to0\right).
\end{equation}

\end{lemma}

\begin{remark}
In this expression, we compare the law of the family of random variables $\left(e^{i\xi(X_O^n+Y_O^n)}\right)_{O\in\mathrm{Per}(n)}$ to the uniform law on the torus of dimension $\# \mathrm{Per}(n)$.  The proof of this lemma is given  in the next subsection.
 
\end{remark}

\subsection{A normal law of large variance on the circle is close to uniform}
We will need the following lemma, which evaluates how much the law $\mc N(0,1)\mod \frac 1 t$ differs from the uniform law on the circle $\mathbb{R}/(\frac{1}{t} \mathbb{Z})$ for large values of $t$.
\begin{lemma}\label{3}
There exists a constant $C>0$ such that for every real numbers $\alpha,\beta$, such that $0<\beta-\alpha<2\pi$ and every real number $t\geq 1$,
\[\left|\int_{\ma R}\sum_{k\in\ma Z}\mathds 1_{\frac{\alpha+2k\pi}t\leq x\leq\frac{\beta+2k\pi}t}e^{-\frac{x^2}{2}}\frac{dx}{\sqrt{2\pi}}-\frac{\beta-\alpha}{2\pi}\right|\leq\frac C t (\beta-\alpha).\]
 \begin{figure}[h] 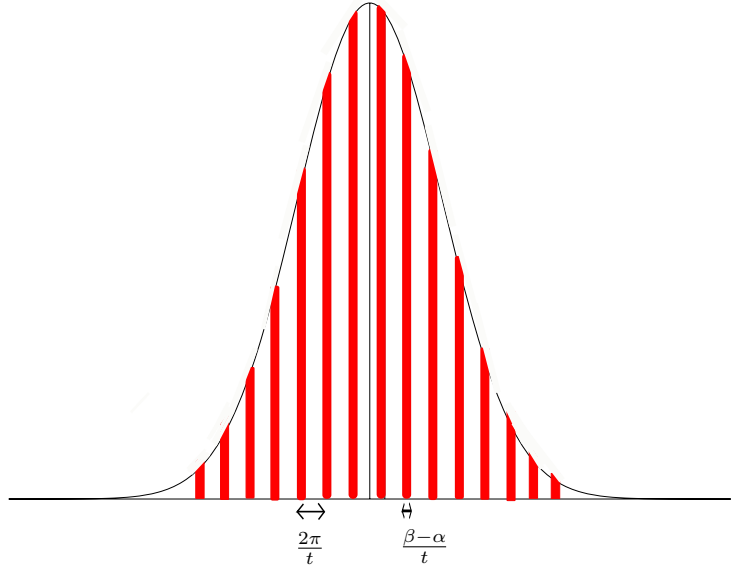 \caption{As $t$ goes to infinity, the red area converges to $\frac{\beta-\alpha}{2\pi}$ with speed $O(\frac {\beta-\alpha} t).$} \label{fig:dessin} \end{figure}
\end{lemma}
\begin{proof}
By mean value inequality, if $|x-y|\leq 1$, then
\[\left|e^{-\frac{x^2}2}-e^{-\frac{y^2}2}\right|\leq |x-y| f(y)\]
for the $L^1$ function  \[f(y):=\sup_{|u-y|\leq 1} |u|e^{-\frac{u^2}2}\]

Let us then write for $u\in[\frac \alpha t,\frac \beta t]$, $u_k:=u+\frac{2k\pi}{t}$ and $I_k:=[u_k,u_{k+1}]$.
We have just seen that for $t\geq 2\pi$, for all $y\in I_k$,
\[\left|e^{-\frac{u_k^2}2}-e^{-\frac{y^2}2}\right|\leq \frac C t f(y).\]
Integrating over $y\in I_k$ of length $\frac{2\pi}{t}$ and summing over $k\in\ma Z$ yields
\[\left|\frac{2\pi}{t}\sum_{k\in\ma Z}e^{-\frac{u_k^2}2}-\sqrt{2\pi}\right|\leq \frac Ct\]
(The value of the constant $C$ changes at each line, but it depends neither on $t$, nor on $\alpha,\beta$.)
Averaging over $u\in[\frac\alpha t,\frac\beta t]$ gives
\[\left|\frac{2\pi}{\beta-\alpha}\sum_{k\in\ma Z}\int_{\frac\alpha t}^{\frac\beta t}\exp\left(-\frac{(u-2k\pi)^2}2\right)du-\sqrt{2\pi}\right|\leq\frac Ct.\]
Consequently,
\[\left|\int_{\ma R}\sum_{k\in\ma Z}\mathds 1_{\frac{\alpha+2k\pi}t\leq x\leq\frac{\beta+2k\pi}t}e^{-\frac{x^2}{2}}\frac{dx}{\sqrt{2\pi}}-\frac{\beta-\alpha}{2\pi}\right|\leq\frac Ct(\beta-\alpha)\]
\end{proof}
\begin{proof}[Proof of lemma \ref{2.5}]

Let us denote by $E$ the expectation
\[E:=\ma E\left[\prod\limits_{m|n}\left(\prod\limits_{O\in\mc P_m}\lambda_O\mathds 1_{A_O}\left(e^{i\xi(X_O^n+Y_O^n)}\right)\right)^m\right].\]
If we write respectively $\ma P_X$, $\ma P_Y$ and $\ma P_{X,Y}$ the probability laws of the variables $(\xi X^n_O)_{O\in\mathrm{Per}(n)}$, $(\xi Y^n_O)_{O\in\mathrm{Per}(n)}$ and  $(\xi X^n_O)_{O\in\mathrm{Per}(n)}\cup(\xi Y^n_O)_{O\in\mathrm{Per}(n)}$ respectively, then  condition (C) of Definition (\ref{C}) implies
\begin{multline}
    d\ma P_{X,Y}((x_O)_{O\in\mathrm{Per}(n)},(y_O)_{O\in\mathrm{Per}(n)}) =\\ \prod\limits_{m|n}\prod\limits_{O\in\mc P_m}e^{-\frac{{x_O}^2}{2\sigma^2_{n,\xi}}}\frac{dx_O}{\sigma_{n,\xi}\sqrt{2\pi}}\otimes d\ma P_Y((y_O)_{O\in\mathrm{Per}(n)}).
\end{multline}
with the variance $\sigma^2_{n,\xi}:=\xi^2\frac{n^2}mK_n(0)$.
We have
\begin{multline}
    E=\int_{\ma R^{2\#\mathrm{Per}(n)}}\prod\limits_{{O\in\mathrm{Per}(n)}}\left(\sum_{k\in\ma Z}\lambda_O^m\mathds 1_{]\alpha_O+2k\pi,\beta_O+2k\pi[}(x_O+y_O)\right)\\d\ma P_{X,Y}((x_O)_{O\in\mathrm{Per}(n)},(y_O)_{O\in\mathrm{Per}(n)}).
\end{multline}
Thus, writing  $u_O=\frac{x_O}{\sigma_{n,\xi}}$ for $O\in\mc P_m$,
\begin{multline}
    E=\int_{\ma R^{\#\mathrm{Per}(n)}}\prod\limits_{m|n}\prod\limits_{O\in\mc P_m}\left(\int_{\ma R}\sum_{k\in\ma Z}\lambda_O^m\mathds 1_{\left]\frac{\alpha_O-y_O+2k\pi}{\sigma_{n,\xi}},\frac{\beta_O-y_O+2k\pi}{\sigma_{n,\xi}}\right [}(u_O)e^{-\frac{u_O^2}2}\frac{du_O}{\sqrt{2\pi}}\right)\\d\ma P_{Y}((y_O)_{O\in\mathrm{Per}(n)}) .
\end{multline}
Let us write for $O\in\mathrm{Per}(n)$
\[I_O = \int_{\ma R}\sum_{k\in\ma Z}\lambda_O^m\mathds 1_{\left]\frac{\alpha_O-y_O+2k\pi}{\sigma_{n,\xi}},\frac{\beta_O-y_O+2k\pi}{\sigma_{n,\xi}}\right [}(u_O)e^{-\frac{u_O^2}2}\frac{du_O}{\sqrt{2\pi}}.\]
Lemma \ref{3} yields
\[I_O=\lambda_O^m\frac{\beta_O-\alpha_O}{2\pi}\left(1+\epsilon_O\right),\]
where
\[\begin{split}
    \exists C>0,\abs{\epsilon_O}&\leq\frac{C}{\sigma_{n,\xi}}\\
    &\leq\frac{C}{\xi\sqrt{nK_n(0)}}.
\end{split}\]
Let us remark that for every finite family $\{x_k\}_k\subset\ma R$, the expansion of the product and factorization after triangular inequality give
\[\left|\prod_k(1+x_k)-1\right|\leq\prod_k(1+|x_k|)-1.\]
Thus, 
\[\begin{split}
    \left|\frac{\prod\limits_{m|n}\prod\limits_{O\in\mc P_m}I_O}{\prod\limits_{m|n}\prod\limits_{O\in\mc P_m}\lambda_O^m\left(\frac{\beta_O-\alpha_O}{2\pi}\right)}-1\right|&=\left|\prod\limits_{m|n}\prod\limits_{O\in\mc P_m}(1+\epsilon_O)-1\right|\\
    &\leq \left(\left(1+\frac{C}{\xi(nK_n(0))^{1/2}}\right)^{\#\mathrm{Per}(n) }-1\right).
\end{split}\]
From Lemma \ref{periodic} we have $\#\mathrm{Per}(n) \leq l^n$.

Using hypothesis (\ref{cond2}) we can bound the prefactor:
\[\begin{split}\left(1+\frac{C}{\xi(nK_n(0))^{1/2}}\right)^{l^n}-1&\underset{(\ref{kj}),(\ref{K(0)})}=\left(1+\frac{CM^{n(k+\frac12+\frac\epsilon2)}}{\xi\sqrt n}\right)^{l^n}-1\\
&\leq\exp(l^n\frac{CM^{n(k+\frac12+\frac\epsilon2)}}{\xi\sqrt n})-1\\
&\leq C'\frac{l^nM^{n(k+\frac12+\frac\epsilon2)}}{\xi\sqrt n}
 \end{split}\]
 for some $C'>0$ for $n$ and $\xi$ large enough  and satisfying (\ref{cond2}) since
 \begin{equation}\label{tendvers0}
     l^n\frac{CM^{n(k+\frac12+\frac\epsilon2)}}{\xi\sqrt n}\underset{(\ref{cond2})}{\leq}C\xi^{c-1}n^{-\frac12}\to 0.
 \end{equation}
\end{proof}
\subsection{End of proof}
We can now easily extend the lemma \ref{2.5} from characteristic functions to step functions.
\begin{corollary}\label{cor}
Assume that  $n$ and $\xi$ satisfy  (\ref{cond2}).  For any families $(X_O^n)_{\begin{subarray}{l}n\geq 1  \\O\in \mathrm{Per}(n)\end{subarray} }$ and $(Y_O^n)_{\begin{subarray}{l}n\geq 1  \\O\in \mathrm{Per}(n)\end{subarray} }$ of real random variables  satisfying condition (C) of Definition (\ref{C}), there exists $C>0$ such that, if $(f_{n,O})_{\begin{subarray}{l}n\geq 1  \\O\in \mathrm{Per}(n)\end{subarray}}$ is a family of step functions $S^1\to\ma R$, then
\begin{multline} \label{BE1} \left|\ma E\left[\prod\limits_{m|n}\prod\limits_{O\in\mc P_m}f_{n,O}^m(e^{i\xi(X_O^n+Y_O^n)})\right]-\prod\limits_{m|n}\prod\limits_{O\in\mc P_m}\int f_{n,O}^md\mathrm{Leb}\right|\\\leq C\xi^{c-1}n^{-\frac12}\prod\limits_{m|n}\prod\limits_{O\in\mc P_m}\int \abs{f_{n,O}^m}d\mathrm{Leb}.\end{multline}
\end{corollary}
\begin{proof}
Let us write each  $f_{n,O}$ as \[f_{n,O}= \sum\limits_{q=1}^{p_{n,O}}\lambda_{n,O,q}\mathds 1_{A_{n,O,q}},\]
where the $\lambda_{n,O,q}$ are complex numbers and the $A_{n,O,q}, 1\leq q \leq p_{n,O}$ are disjoint intervals.  We develop (\ref{BE1}), we use Lemma \ref{2.5} and factorize the result:
\begin{multline*}E:=\ma E\left[\prod\limits_{m|n}\prod\limits_{O\in\mc P_m}f_{n,O}^m(e^{i\xi(X_O^n+Y_O^n)})\right]\\=\sum_{(q_O)\in\prod\limits_{m|n}\prod\limits_{O\in\mc P_m}\{1,\cdots,p_{n,O}\}}\ma E\left[\prod\limits_{m|n}\prod\limits_{O\in\mc P_m}\lambda_{n,O,q_O}^m\mathds 1_{A_{n,O,q_O}}(e^{i\xi(X_O^n+Y_O^n)})\right].\end{multline*}
Consequently,
\begin{multline*}
    \left|E-\prod\limits_{m|n}\prod\limits_{O\in\mc P_m}\int f_{n,O}^md\mathrm{Leb}\right|\\\leq\sum_{(q_O)\in\prod\limits_{m|n}\prod\limits_{O\in\mc P_m}\{1,\cdots,p_{n,O}\}}\left|\ma E\left[\prod\limits_{m|n}\prod\limits_{O\in\mc P_m}\lambda_{n,O,q_O}^m\mathds 1_{A_{n,O,q_O}}(e^{i\xi(X_O^n+Y_O^n)})\right]\right.-\\\left.\prod_{m|n}\prod_{O\in\mc P_m}\lambda_{n,O,q_O}^m\text{Leb}(A_{n,O,q_O})\right|\\\leq C\xi^{c-1}n^{-\frac12}\sum_{(q_O)\in\prod\limits_{m|n}\prod\limits_{O\in\mc P_m}\{1,\cdots,p_{n,O}\}}\prod_{m|n}\prod_{O\in\mc P_m}\abs{\lambda_{n,O,q_O}}^m\text{Leb}(A_{n,O,q_O})
\end{multline*}
from the previous lemma. \\
Hence,
\[\left|E-\prod\limits_{m|n}\prod\limits_{O\in\mc P_m}\int f_{n,O}^md\mathrm{Leb}\right|\leq C\xi^{c-1}n^{-\frac12}\prod\limits_{m|n}\prod\limits_{O\in\mc P_m}\int \abs{f_{n,O}}^md\mathrm{Leb}.\]
\end{proof}
We can use this result in order to estimate the characteristic function of $\mathrm{Tr}^\flat(\mc L^n_{\xi,\tau})$, using remark (\ref{Char_function}).
\begin{corollary}
\label{pC}
Assume that  $n$ and $\xi$ satisfy (\ref{cond2}).  Let $(X_O^n)_{\begin{subarray}{l}n\geq 1  \\O\in \mathrm{Per}(n)\end{subarray} }$ and $(Y_O^n)_{\begin{subarray}{l}n\geq 1  \\O\in \mathrm{Per}(n)\end{subarray} }$ be two families of real random variables satisfying condition (C) of Definition (\ref{C}).  There exists $C>0$ such that for all $(\mu_O,\nu_O)_{O\in\mathrm{Per}(n)}\in\ma R^{2\#\mathrm{Per}(n)}$,
\begin{multline*}\left|\ma E\left[\prod\limits_{m|n}\prod\limits_{O\in\mc P_m}e^{im\mu_O\cos(\xi(X_O^n+Y_O^n))+im\nu_O\sin(\xi(X_O^n+Y_O^n))}\right]\right.\\\left.-\prod\limits_{m|n}\prod\limits_{O\in\mc P_m}\int_0^{2\pi}e^{i(m\mu_O\cos\theta+m\nu_O\sin\theta)}\frac{d\theta}{2\pi}\right|\leq C\xi^{c-1}n^{-\frac12}.\end{multline*}

\end{corollary}
\begin{proof}
Let $C$ be the constant from corollary \ref{cor}.
For $O\in\mathrm{Per}(n)$, let $f_{O}$ be the function defined on $S^1$ by
\[f_{O}(e^{i\theta})=e^{i(\mu_O\cos\theta+\nu_O\sin\theta)}.\]
Each $f_{O}$ is bounded by $1$, we can consequently find for each $O\in\mathrm{Per}(n)$ a family $(f_{j,O})_j$  of step functions uniformly bounded by $1$ converging pointwise towards $f_{O}$.\newline
We have for $n$ fixed, by dominated convergence \[E_j:=\ma E\left[\prod\limits_{m|n}\prod\limits_{O\in\mc P_m}f_{j,O}^m(e^{i\xi(X^n_{O}+Y^n_{O})})\right]\xrightarrow[j\to\infty]{}E:=\ma E\left[\prod\limits_{m|n}\prod\limits_{O\in\mc P_m}f^m(e^{i\xi(X^n_{O}+Y^n_{O})})\right]\]
as well as
\[I_j:=\prod\limits_{m|n}\prod\limits_{O\in\mc P_m}\int_0^{2\pi}f_{j,O}^m(e^{i\theta})\frac{d\theta}{2\pi}\xrightarrow[j\to\infty]{}I:=\prod\limits_{m|n}\prod\limits_{O\in\mc P_m}\int_0^{2\pi}f_{O}^m(e^{i\theta})\frac{d\theta}{2\pi}.\]
It is thus possible to find an integer $j_0$ such that both
\[\left|E-E_{j_0}\right|\leq \xi^{c-1}n^{-\frac12}\]
and
\[\left|I-I_{j_0}\right|\leq\xi^{c-1}n^{-\frac12}\]
hold.\newline
From corollary \ref{cor}, we know that for all $n\in\ma N$
\[\left|E_{j_0}-I_{j_0}\right|\leq C\xi^{c-1}n^{-\frac12}\sup\abs{f_{j_0}}.\]
Thus,
\[\begin{split}\abs{E-I}&\leq\left|E-E_{j_0}\right|+\left|E_{j_0}-I_{j_0}\right|+\left|I-I_{j_0}\right|\\
                        &\leq (C+2)\xi^{c-1}n^{-\frac12}.\end{split}\]
\end{proof}
We can know prove the final proposition :
\begin{proposition}\label{2.8}Let $(X_O^n)_{\begin{subarray}{l}n\geq 1  \\O\in \mathrm{Per}(n)\end{subarray} }$ and $(Y_O^n)_{\begin{subarray}{l}n\geq 1  \\O\in \mathrm{Per}(n)\end{subarray} }$ be two families of real random variables  satisfying condition (C) of Definition (\ref{C}).
If condition (\ref{cond2}) is satisfied then we have the following convergence in law
\begin{equation}
\label{Tnxi}
T_{n,\xi} := A_n\sum_{m|n}m\sum_{O\in\mc P_m}\frac{e^{i\xi(X^n_O+Y^n_O)}}{e^{J^n_O}-1}\underset{n,\xi\to\infty}{\longrightarrow}\mc N_{\ma C}(0,1 ),    
\end{equation}
with the amplitude $A_n$  defined  in (\ref{def_An}) by 
\begin{equation}\label{An}
    \begin{split}A_n&=\left(\sum_{m|n}m^2\sum_{O\in\mc P_m}\frac 1{(e^{J_O^n}-1)^2}\right)^{-\frac 12}.\end{split}
\end{equation}
\end{proposition}

\begin{proof}
Let us fix two real numbers $\xi_1$ and $\xi_2$ and  let  $\phi_n$ be the characteristic function of $T_{n,\xi}$:
\begin{multline*}
    \phi_n(\xi_1,\xi_2) :=\ma E\left[\exp\left(iA_n\left(\xi_1\sum\limits_{m|n}m\sum\limits_{O\in\mc
P_m}\frac{\cos(\xi(X_O^n+Y_O^n))}{e^{J^n_O}-1}+\right.\right.\right.\\\left.\left.\left.\xi_2\sum\limits_{m|n}m\sum\limits_{O\in\mc P_m}\frac{\sin(\xi(X^n_O+Y^n_O))}{e^{J^n_O}-1}\right)\right)\right].
\end{multline*}
We compute the limit of $\phi_n(\xi_1,\xi_2)$  as $n$ goes to infinity.
Corollary (\ref{pC}) yields
\begin{equation}
    \label{phi}\left|\phi_n(\xi_1,\xi_2)-\prod\limits_{m|n}\prod\limits_{O\in\mc P_m}\int_0^{2\pi}e^{i\frac{mA_n}{e^{J^n_O}-1}(\xi_1\cos\theta+\xi_2\sin\theta)}\frac{d\theta}{2\pi}\right|\leq C\xi^{c-1}n^{-\frac12}\to 0
\end{equation}
under the assumption (\ref{cond2}).

Let 
\[
\psi(\xi_1,\xi_2) := \int_0^{2\pi} e^{i(\xi_1\cos\theta+\xi_2\sin\theta)}\frac{d\theta}{2\pi}.\]
We have the following Taylor's expansion in 0:
\[\psi(\xi_1,\xi_2) = 1-\frac 1 4 (\xi_1^2+\xi_2^2)+o(\xi_1^2+\xi_2^2).\]
In order to apply this to equation (\ref{phi}), we need to check that  
\begin{lemma}\label{ref}
\[nA_n\sup_{O\in\mathrm{Per}(n)}\frac{1}{e^{J^n_O}-1}\underset{{n\to\infty}}{\longrightarrow}0.\]
\end{lemma}                                                                                                                                                                              
\begin{proof}
See appendix \ref{label}
\end{proof}
We can now state that
\[\begin{split}
    &\prod\limits_{m|n}\prod\limits_{O\in\mc P_m}\int_0^{2\pi}e^{i\frac{mA_n}{e^{J_O}-1}(\xi_1\cos\theta+\xi_2\sin\theta)}\frac{d\theta}{2\pi}=\prod\limits_{m|n}\prod\limits_{O\in\mc P_m}\psi\left(\xi_1\frac{mA_n}{e^{J_O}-1},\xi_2\frac{mA_n}{e^{J_O}-1}\right)\\
    &=\prod\limits_{m|n}\prod\limits_{O\in\mc P_m}\left( 1-\frac {\xi_1^2+\xi_2^2} 4\frac{(mA_n)^2}{(e^{J_O}-1)^2} +o\left(\frac{(mA_n)^2}{(e^{J_O}-1)^2}\right)\right)\\
&=\exp\left(\sum_{m|n}\sum_{O\in\mc P_m}\log\left(1-\frac {\xi_1^2+\xi_2^2} 4\frac{(mA_n)^2}{(e^{J_O}-1)^2} +o\left(\frac{(mA_n)^2}{(e^{J_O}-1)^2}\right)\right)\right)\\
&=\exp\left(\sum_{m|n}\sum_{O\in\mc P_m}-\frac {\xi_1^2+\xi_2^2} 4\frac{(mA_n)^2}{(e^{J_O}-1)^2} +o(\frac{(mA_n)^2}{(e^{J_O}-1)^2})\right)\\
& \underset{(\ref{An})}= e^{-\frac{\xi_1^2+\xi_2^2}4+o(1)},\end{split}\]
We deduce that \[\phi_n(\xi_1,\xi_2)\underset{n\to\infty}{\longrightarrow}e^{-\frac{\xi_1^2+\xi_2^2}4}\]
which is the characteristic function of a Gaussian variable of law $\mc N_{\ma C}(0,1).$
\end{proof}

\section{Discussion}
In this paper we have considered a model where the roof function $\tau$ is random.  However, the numerical experiments suggest a far stronger result: for a fixed function $\tau$ and a semiclassical parameter $\xi$ chosen according to a uniform random distribution in a small window at high frequencies, the result seems to remain true, as shown in the following figures for $\tau(x) =\sin(2\pi x)$.  The moduli also seem to become uniform.   \newline

It would be interesting to understand what informations about the Ruelle resonances can be recovered from the convergence (\ref{theorem}).  We know from the Weyl law from \cite{arnoldi2017asymptotic} established in a similar context that the number of resonances of $\mc L_{\xi,\tau}$ outside the essential spectral radius, for a given $\tau$, are of order $O(\xi)$.  A complete characterization would thus require a knowledge of the traces of $\mc L_{\xi,\tau}^n$ up to times of order $O(\xi)$, while we only have information for $n=O(\log\xi).$

 \begin{figure}[h!]  \includegraphics[scale=0.4]{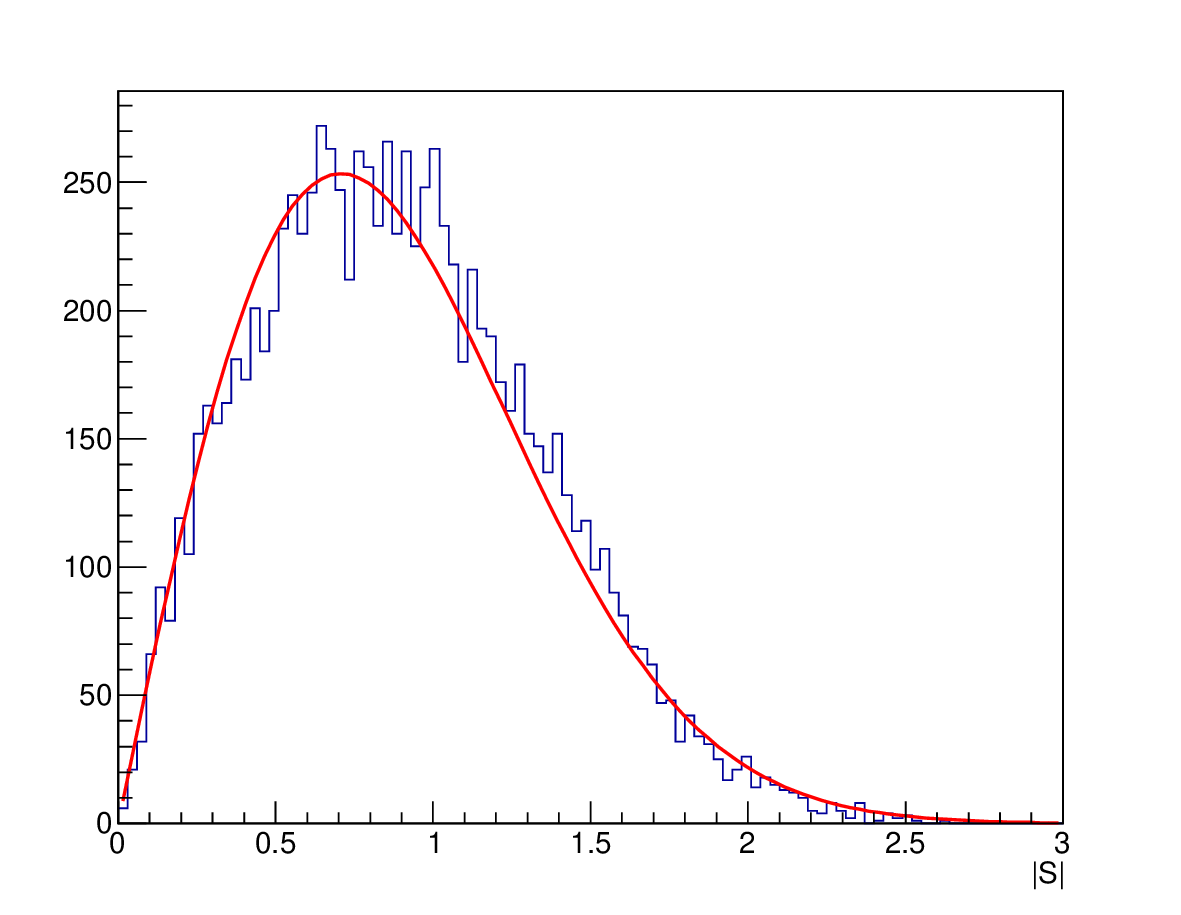}
  \caption{Histogramm of $ S = \left| A_n\mathrm{Tr}^\flat\left(\mc L^n_{\tau_0,\xi}\right) \right| $ for a sample of $10^4$ random values of $\xi$ uniformly distributed in $[\xi_0, \xi_0 +10]$ with $\xi_0= 2.10^6$ and $n=11$ corresponding to a fraction of the Ehrenfest time $C_e:= n\frac{\log2}{\log \xi_0}=0.5$.  It  is well fitted by  the red curve $S\mapsto  C S \exp (-S^2) $.} \label{fig:dessin3} \end{figure}

 \begin{figure}[h!]  \includegraphics[scale=0.4]{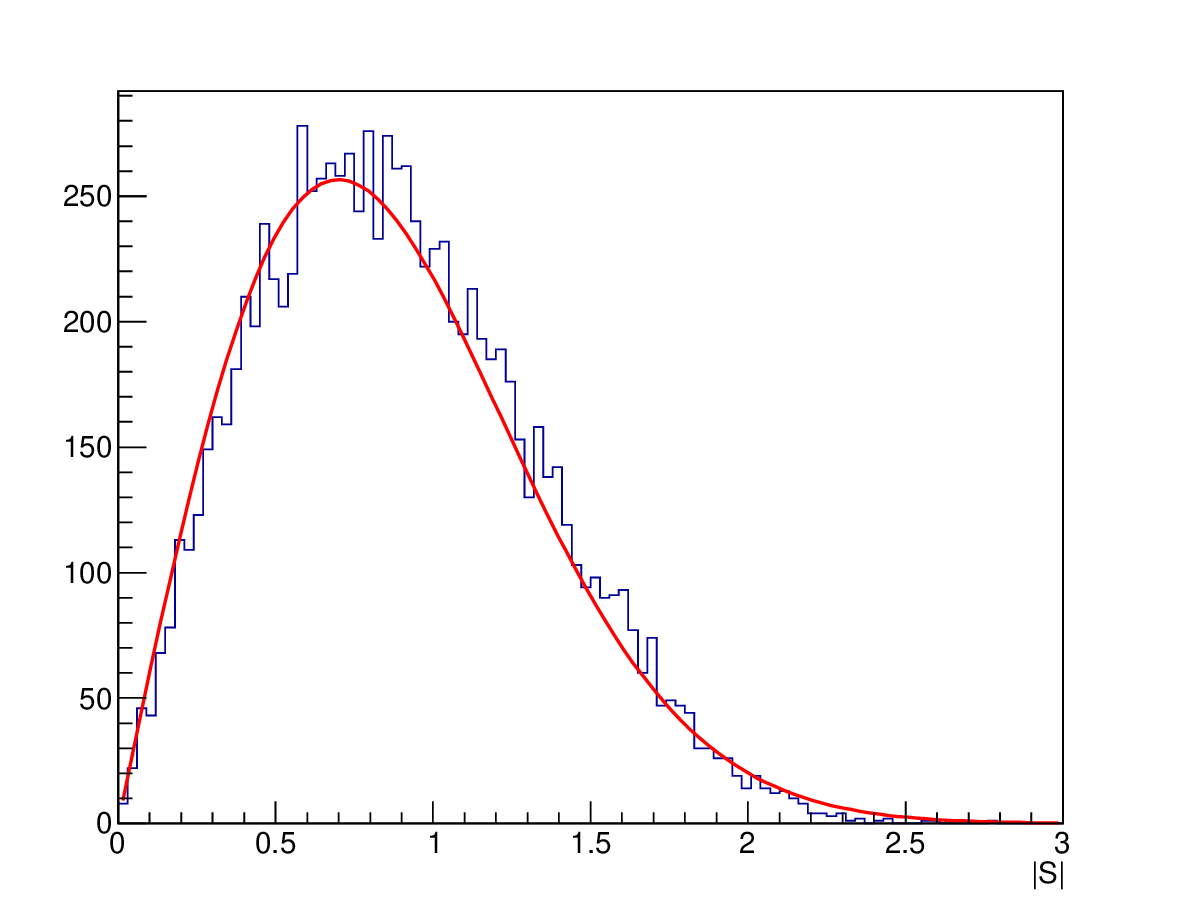}
  \caption{Histogramm of $ S = \left| A_n\mathrm{Tr}^\flat\left(\mc L^n_{\tau_0,\xi}\right) \right| $ for a sample of $10^4$ random values of $\xi$ uniformly distributed in $[\xi_0, \xi_0 +10]$ with $\xi_0= 2000$ and $n=11$ giving  $C_e=1.0$.  The red curve corresponds to $S\mapsto  C S \exp (-S^2) $.} \label{fig:dessin3} \end{figure}
\newpage
\appendix
\section{Proof of lemma \ref{periodic}}\label{annexe1}
 \begin{lemma}\label{periodic}
For every integer $n$, $E^n$ has $l^n-1$ fixed points.
The distance between two distinct periodic points is bounded from below by $\frac 1{M^n-1}$.
 \end{lemma}
 \begin{proof}
$E$ is topologically conjugated to the linear expanding map of same degree $x\mapsto lx\mod 1$, (see \cite{katok1997introduction}, p.73).
Thus $E^n$ has $l^n-1$ fixed points.
Let $\tilde E:\ma R\longrightarrow \ma R$ be a lift of $E$, $x\neq y$ be two fixed points of $E^n$ and $\tilde x,\tilde y\in\ma R$ be representatives of $x$ and $y$ respectively.  Note that 
\[d(x,y)=\inf|\tilde x-\tilde y|\]
where the infimum is taken over all couples of representatives $(\tilde x,\tilde y)$.
Since $E^n(x) = x$ and $E^n(y) = y$, $\tilde E^n(\tilde y)-\tilde E^n(\tilde x)-(\tilde y-\tilde x)$ is an integer, different from $0$ because $\tilde E^n $ is expanding.
Thus,
\[\left|\tilde E^n(\tilde y)-\tilde E^n(\tilde x)-(\tilde y-\tilde x)\right|\geq 1,\]
that is
\[\left|\int_{\tilde x}^{\tilde y}\left((\tilde E^n)'(t)-1\right)\mathrm{d}t\right|\geq 1\]
Finally,
\[|\tilde y-\tilde x|(M^n-1)\geq 1.\]
Taking the infimum gives the result.
\end{proof}
\section{Proof of lemma \ref{regularite} on the link between regularity of a Gaussian field and variance of the Fourier coefficients}\label{Borel-Cantelli}
Let us recall the following classical estimate:
\begin{lemma}\label{lemme Alejandro}
If $(X_p)_{p\in\ma Z}$ is a family of independent centered Gaussian random variables of variance $1$, then, almost surely,
\[\forall \delta>0,X_p=o(p^\delta).\]

\end{lemma}
\begin{proof}
Let $\delta>0$.  Let us use Borel-Cantelli lemma:
\[\forall p\in\ma Z, \ma P(\abs{X_p}>p^{\delta})= \int_{\abs x>p^{\delta}}e^{-\frac{x^2}2}\frac {dx}{\sqrt{2\pi}}.\]
Now, we have the upper bound
\[p^{\delta}\int_{p^{\delta}}^{+\infty}e^{-\frac{x^2}2}\frac {dx}{\sqrt{2\pi}}\leq\int_{p^{\delta}}^{+\infty}xe^{-\frac{x^2}2}\frac {dx}{\sqrt{2\pi}} =\frac{e^{-\frac {p^{2\delta}}2}}{\sqrt{2\pi}}.\]
Thus,
\[\forall p\in\ma Z^*, \ma P(\abs{X_p}>p^{\delta})\leq\frac 2{\sqrt{2\pi}p^{\delta}} e^{-\frac {p^{2\delta}}2}.\]
Consequently,
\[\sum_p\ma P(\abs{X_p}>p^{\delta})<\infty\] and by Borel-Cantelli, almost surely, \[\#\{p\in\ma Z,\abs{X_p}>p^{\delta}\}<\infty.\]
\end{proof}
With this in mind, we can see that if a real random function $\tau$ has random Fourier coefficients $(c_p)_{p\in\ma Z}$, pairwise independent (for non-negative values of $p$), with variance
\[\sigma_p^2:=\ma E\left[\abs{c_p}^2\right]= O(\frac 1{p^{2k+2+\eta}}),\]
for some $\eta>0$,
then by the previous lemma, almost surely, for all $\delta>0$,
\[\frac{ c_p}{\sigma_p}= o(p^\delta),\] and thus for $\delta = \frac\eta 2,$
\[ c_p = o\left(\frac 1 {p^{k+1+\frac\eta 2}}\right)\ \text{a.s}.\]
As a consequence,\[\sum_p c_p (2i\pi p)^k e^{2i\pi px}\]converges normally and thus $\tau$ is almost surely $\mc C^k$.

\section{Ruelle resonances and Flat trace}\label{appendice Ruelle spectrum}

\subsection{Ruelle spectrum}If $\tau\in\mc C^k(\ma T)$, the operator $\mc L_{\xi,\tau}$ can be extended to distributions $(\mc C^k(\ma T))'$ by duality.  We will denote $H^{s}(\ma T)$ the Sobolev space of order $s \in \ma R$.
\begin{theorem}[\cite{ruelle1986locating},\cite{baladi2018dynamical} Thm 2.15 and Lemma 2.16]
Let $k\geq1$.  If $\tau$ belongs to $\mc C^k$, then for every $0\leq s<k,$ $\mc L_{\xi,\tau}:H^{-s}(\ma T)\rightarrow H^{-s}(\ma T)$ is bounded and its essential spectral radius $r_{\mathrm{ess}}$ satisfies
\[r_{\mathrm{ess}}\leq \frac{e^{\mathrm{Pr}(-\frac 12J)}}{m^s},\]
where $m=\inf E'$, $J(x)=\log E'(x)$ and $\mathrm{Pr}(-\frac 12J)$ is defined in \ref{pressure}.
\end{theorem}
The  discrete set of eigenvalues of finite multiplicities outside a given disk of radius $r\geq \frac{e^{\mathrm{Pr}(-\frac 12J)}}{m^s}$, and the associated eigenspaces remain the same in every space $H^{-s'}(\ma T)$ for $s'\geq s$.  This can be deduced for example from the fact that these spectral elements give the asymptotic behaviour of the correlation functions: for any smooth functions $f,g$ on $\ma T$, for any $s$ large enough, if $\mc L_{\xi,\tau}:H^{-s}(\ma T)\rightarrow H^{-s}(\ma T)$ has no eigenvalue of modulus $r$,  
\begin{equation}
    \int \mc L^n_{\xi,\tau}f\cdot g = \sum_{\substack{\lambda\in\sigma(\mc L_{\xi,\tau})\\|\lambda|>r}}\int \mc L^n_{\xi,\tau}(\Pi_\lambda f)\cdot g +O_{n\to\infty}(r^n),
    \end{equation}
    where $\Pi_\lambda$ is the spectral projector associated to $\lambda.$
 We are interested in the statistical properties of these eigenvalues, called Ruelle-Pollicott spectrum or Ruelle resonances, when $\tau$ is a random function.  One way to get informations about the spectrum of such operators is using a trace formula.
Although $\mc L_{\xi,\tau}$ is not trace-class, we can give a certain sense to the trace of $\mc L_{\xi,\tau}$.
\subsection{Flat trace}
This section is an adaptation of section 3.2.2 in \cite{baladi2018dynamical}
In order to motivate the definition of flat trace, let us first recall the following fact:
\begin{lemma}Let $m> \frac12$. (Then the Dirac distributions belong to $H^{-m}(\ma T)$).
If $T:H^{-m}(\ma T)\longrightarrow H^m(\ma T)$ is a bounded operator,  then it has a continuous Schwartz kernel $K$ and 
\[K(x,y)=\langle\delta_x,T\delta_y\rangle.\]
If moreover $T$ is class-trace, then 
\[\mathrm{Tr}\ T=\int_{\ma T}K(x,x)\mathrm{d}x.\]
\end{lemma}
Let $\rho$ be a smooth compactly supported function such that $\int_{\ma R}\rho=1$.  For $\epsilon>0$ and $y\in\ma T$ we write
\[\rho_{\epsilon,y}(t)=\frac1\epsilon\rho\left(\frac{t-y}{\epsilon}\right).\]Periodizing this function gives rise to a smooth function $\rho_{\epsilon,y}$ on $\ma T$ satisfying
\[\rho_{\epsilon,y}\underset{\epsilon\to0}\longrightarrow \delta_y\] as distributions. 
\begin{defi}
Let $m\geq 0$ and $T:H^{-m}(\ma T)\longrightarrow H^{-m}(\ma T)$ be a bounded operator extending to a continuous operator $(\mc C^0(\ma T))'\longrightarrow(\mc C^0(\ma T))'$.  Then the formula \[K_\epsilon(x,y):=\langle\rho_{\epsilon,x},T\delta_y\rangle\]
defines for every $\epsilon>0$ a continuous function on $\ma T^2.$ Let
\[\mathrm{Tr}^\flat_\epsilon(T):=\int_{\ma T}K_\epsilon(x,x)\mathrm{d}x.\]
We say that $T$ admits a flat trace $\mathrm{Tr}^\flat(T)$ if $\mathrm{Tr}^\flat_\epsilon(T)\rightarrow\mathrm{Tr}^\flat(T)$ as $\epsilon$ goes to zero, independently of the choice of the mollifying function $\rho$.
\end{defi}
Note that, for any $n\in\ma N^*,$ $\xi\in\ma R$, $\tau\in\mc C^0(\ma T)$, the transfer operator $\mc L^n_{\xi,\tau}:(\mc C^0(\ma T))'\longrightarrow(\mc C^0(\ma T))'$ is bounded. 
\begin{lemma}[Trace formula, \cite{atiyah1967lefschetz}, \cite{guillemin1977lectures}]
Let $\tau\in\mc C^k(\ma T), k\geq0.$ For any integer $n\geq 1$,  $\mc L_{\xi,\tau}^n$ has a flat trace
\begin{equation}\mathrm{Tr}^\flat\mc L^n_{\xi,\tau}= \sum_{x,E^n(x)=x}\frac{e^{i\xi\tau_x^n}}{{(E^n)'(x)}-1}
\end{equation}
\end{lemma}
\begin{proof}
\[\mathrm{Tr}^\flat_\epsilon(\mc L^n_{\xi,\tau})=\int_{\ma T}\langle\rho_{\epsilon,x},\mc L^n_{\xi,\tau} \delta_x\rangle\mathrm dx.\]
By definition of the action of $\mc L^n_{\xi,\tau}$ on distributions,
\[\langle\rho_{\epsilon,x},\mc L^n_{\xi,\tau} \delta_x\rangle=(\mc L^n_{\xi,\tau})^*\rho_{\epsilon,x}(x),\]
where $(\mc L^n_{\xi,\tau})^*$ is the $L^2$-adjoint of $\mc L^n_{\xi,\tau}.$
Let us recall that, if $\phi:\ma T\longrightarrow\ma T$ is a local diffeomorphism, for every continuous functions $u,v$ on $\ma T$,
\begin{equation}
    \label{changement de variables}
    \int u(\phi(y))v(y)\mathrm{d}y=\int u(x)\sum_{\phi(y)=x}\frac{v(y)}{|\phi'(y)|}\mathrm{d}x.
\end{equation}
Thus, \[(\mc L^n_{\xi,\tau})^*v(x)=\sum_{E^n(y)=x}\frac{v(y)e^{i\xi\tau^n_y}}{(E^n)'(y)}.\]
Therefore
\[\begin{split}
    \mathrm{Tr}^\flat_\epsilon(\mc L^n_{\xi,\tau})&=\int_{\ma T}(\mc L^n_{\xi,\tau})^*\rho_{\epsilon,x}(x)\mathrm{d}x\\
    &=\int_{\ma T}\sum_{E^n(y)=x}\frac{\rho_{\epsilon,0}(y-E^n(y))e^{i\xi\tau^n_y}}{(E^n)'(y)}\mathrm{d}x\\
    &=\int_{\ma T}\rho_{\epsilon,0}(y-E^n(y))e^{i\xi\tau^n_y}\mathrm{d}y
\end{split}\]
by the change of variables $x=E^n(y)$.
Now, since $E$ is expansive, $y\mapsto y-E^n(y)$ is a local diffeomorphism, so applying (\ref{changement de variables}) once again gives
\[\begin{split}
    \mathrm{Tr}^\flat_\epsilon(\mc L^n_{\xi,\tau})&=\int_{\ma T}\rho_{\epsilon,0}(z)\sum_{y-E^n(y)=z}\frac{e^{i\xi\tau_y^n}}{{(E^n)'(y)}-1} \mathrm{d}z\\
    &\underset{\epsilon\to0}\longrightarrow\sum_{E^n(y)=y}\frac{e^{i\xi\tau_y^n}}{{(E^n)'(y)}-1}.
\end{split}\]
\end{proof}
If $E$ and $\tau$ are analytic, it is well known that $\mc L$ is trace-class and that $\mathrm{Tr}^\flat(\mc L_{\xi,\tau})=\mathrm{Tr}(\mc L_{\xi,\tau})$ (see for instance \cite{jezequel2017local}).  In the smooth setting however the decay of the Ruelle-Pollicott spectrum can be arbitrarily slow (\cite{jezequel2017local}, Proposition 1.10).  The flat trace is however related to the Ruelle-Pollicott spectrum defined above in the following way (This is a consequence of Thm 3.5 in \cite{baladi2018dynamical} and Thm 2.4 in \cite{jezequel2017local}):

\begin{proposition}\label{baladi}
Assume that $\tau\in\mc C^k(\ma T)$ for some $k\geq 1$.
 Let $\xi\in\ma R$, $0\leq s<k$, and $r>\frac{e^{\mathrm{Pr}(-\frac12J)}}{m^s}$ be such that $\mc L_{\xi,\tau}:H^{-s}(\ma T)\longrightarrow H^{-s}(\ma T)$ has no eigenvalue of modulus $r$, then
\begin{equation}
\label{baladi_eq}
   \exists C>0,\forall n\in \ma N,\left|\mathrm{Tr}^\flat\mc L^n_{\xi,\tau} - \sum\limits_{\substack{\lambda\in\sigma(\mc L_{\xi,\tau})\\\abs{\lambda}>r}}\lambda^n \right|\leq Cr^n,
\end{equation}
where the eigenvalues are counted with multiplicity.

\end{proposition}

\section{Proof of lemma \ref{indep}}\label{annexe3}
\begin{proof}

Let $X,X',Y,$ and $Y'$ be as in the statement of the lemma real random variables such that $e^{iX}, e^{iX'}$ are uniform on $S^1$ and so that $X$ ad $ X'$ are both independent of all three other random variables.  Let us write $\ma P_Z$ the law of a random variable $Z$.
To show that $e^{i(X+Y)}$ and $e^{i(X'+Y')}$ are independent and uniform on $S^1$, it suffices to show that for any continuous functions $f,g:S^1\longrightarrow \ma R$,
\[\ma E\left[f(e^{i(X+Y)})g(e^{i(X'+Y')})\right]=\int_0^{2\pi}\int_0^{2\pi}f(e^{i\theta})g(e^{i\theta'})\frac{d\theta}{2\pi}\frac{d\theta'}{2\pi}.\]
\[\ma E\left[f(e^{i(X+Y)})g(e^{i(X'+Y')})\right]=\int_{(S^1)^4}f(e^{i(x+y)})g(e^{i(x'+y')})d\ma P_{(X,Y,X',Y')}(x,y,x',y').\]
By hypothesis,
\[d\ma P_{(X,Y,X',Y')}(x,y,x',y')=\frac{dx}{2\pi}\frac{dx'}{2\pi}d\ma P_{(Y,Y')}(y,y').\]
Thus,
\begin{multline*}
\ma E\left[f(e^{i(X+Y)})g(e^{i(X'+Y')})\right]\\= \int_{(S^1)^2}\left(\int_0^{2\pi}\int_0^{2\pi}f(e^{i(x+y)})g(e^{i(x'+y')})\frac{dx}{2\pi}\frac{dx'}{2\pi}\right)d\ma P{(Y,Y')}(y,y')\\
\underset{\theta = x+y,\theta'=x'+y'}{=}\int_{(S^1)^2}\left(\int_0^{2\pi}\int_0^{2\pi}f(e^{i\theta})g(e^{i\theta'})\frac{d\theta}{2\pi}\frac{d\theta'}{2\pi}\right)d\ma P{(Y,Y')}(y,y')\\
=\int_0^{2\pi}\int_0^{2\pi}f(e^{i\theta})g(e^{i\theta'})\frac{d\theta}{2\pi}\frac{d\theta'}{2\pi}.
\end{multline*}
\end{proof}
\section{Topological pressure}\label{annexe}

\subsection{Definition}

\begin{defi}\label{pressure}
Let $\phi: \ma T\longrightarrow \ma R$ be a Hölder-continuous function. 
The limit 
\begin{equation}
\label{def_Pr}
\mathrm{Pr}(\phi):=\lim_{n\to\infty}\frac 1 n\log\left(\sum_{E^n(x)=x}e^{\phi_x^n}\right)
\end{equation}
exists and is called the topological pressure of $\phi$ (see \cite{katok1997introduction} Proposition 20.3.3 p.630). 

\end{defi}
In other words
\begin{equation}\label{pression}
    \sum_{E^n(x)=x}e^{\phi_x^n}=e^{n\mathrm{Pr}(\phi)+o(n)}.
\end{equation}
The particular case $\phi=0$ gives the topological entropy  $\mathrm{Pr}(0)=h_{top}.$
\begin{remark}\label{absorption}
Note that the expression $e^{n\mathrm{Pr}(\phi)+o(n)}$ describes a large class of sequences, since for instance for any $k\in\ma N$,
\[n^k e^{n\mathrm{Pr}(\phi)}=e^{n\mathrm{Pr}(\phi)+o(n)}.\]
\end{remark}
\subsection {Variational principle}
Another definition of the pressure is given by the variational principle.
Let us denote by $h(\mu)$ the entropy of a measure $\mu$ invariant under $E$ (see \cite{katok1997introduction} section 4.3 for a definition of entropy).
For the next theorem, see \cite{katok1997introduction}, sections 20.2 and 20.3.  The last sentence comes from Proposition 20.3.10.
\begin{theorem}[Variational principle]
Let $\phi: \ma T\longrightarrow \ma R$ be a Hölder function. 
\[\mathrm{Pr}(\phi) = \sup_{\mu\ E-\mathrm{invariant}}\left(\int\phi\ d\mu+h(\mu)\right).\]
This supremum, taken over the invariant \textbf{probability} measures, is moreover attained for a unique $E$-invariant measure $\mu$, called equilibrium measure.  In addition, if we note $ J = \log E'$ and $\mu_\beta$ the equilibrium measure of $-\beta J$, $\beta\mapsto\mu_\beta$ is one-to-one.
\end{theorem}

\begin{corollary}
\label{F}

The function
\begin{equation}
\label{def_F}
\fonction{F}{\ma R_+^*}{\ma R}{\beta}{\frac 1\beta\mathrm{Pr}(-\beta J )}
\end{equation}
is strictly decreasing.
\end{corollary}

\begin{proof}
Let $\beta'>\beta>0$.
By the previous theorem, with the same notations,
\[\int-\beta J\ d\mu_\beta+h(\mu_\beta)>\int-\beta J\ d\mu_{\beta'}+h(\mu_{\beta'})\]
and thus
\[F(\beta) =\int-J\ d\mu_\beta+\frac{h(\mu_\beta)}\beta> \int-J\ d\mu_{\beta'}+\frac{h(\mu_{\beta'})}\beta\geq\int-J\ d\mu_{\beta'}+\frac{h(\mu_{\beta'})} {\beta'}=F(\beta').\]
\end{proof}
 
\subsection{Proof of Lemma \ref{ref}}\label{label}
Let $\phi:\ma T\rightarrow\ma R$ be a $\mc C^1$ function.  Let as before $\phi^n_x$ be the Birkhoff sum (\ref{birkhoff}).  By subadditivity of the sequence $\left(\inf_{x\in\ma T}\phi^n_x\right)_n$ and Fekete's Lemma we can define the following quantity:
\begin{defi}
Let us define  \begin{equation}\label{phimin}
    \phi_{\text{min}}:=\lim_{n\to\infty}\inf_{x\in\ma T}\frac 1n\phi^n_x.
\end{equation}

\end{defi}
\begin{lemma}
The infimum in (\ref{phimin}) can be taken over periodic points:
\begin{equation}
    \phi_{\text{min}}=\lim_{n\to\infty}\inf_{x,E^n(x)=x}\frac 1n\phi^n_x.
\end{equation}
\end{lemma}
\begin{proof}
By lifting the expanding map to $\ma R$, we easily see that $E$ has at least a fixed point $x_0$.  This point has $l^n$ preimages by $E^n$, defining $l^n-1$ intervals $I_k^n$ such that for all $1\leq k\leq l^n-1$
\[E^n:I^n_k\rightarrow\ma T\backslash\{x_0\}\]
is a diffeomorphism.  Thus, there exists $C>0$ such that for all $k$, if $x,y\in \overline {I^n_k}$,
\[\forall0\leq j\leq n,d(E^j(x),E^j(y))\leq \frac C{m^{n-j}},\]
with $m=\inf\abs{E'}>1$.  Each $\overline {I^n_k}$ contains moreover a periodic point $y_{k,n}$ of period $n$ given by $E^n(y_{k,n})=y_{k,n}+k$.
Hence let $n\in\ma N$, let $x_n\in\ma T$ be such that
\[\phi^n_{x_n}= \inf_{x\in\ma T}\phi_x^n,\]
and suppose that $x_n \in \overline{I_k^n}$.  We have
\[\begin{split}
    \left|\phi^n_{y_{k,n}}-\phi^n_{x_n}\right|&=\left|\sum_{j=0}^{n-1}\phi(E^j(x_n))-\phi(E^j(y_{k,n}))\right|\\
    &\leq C\max\abs{\phi'}\sum_{k=0}^\infty \frac{1}{m^k}
\end{split}\]
is bounded independently of $n$.  Consequently
\[\lim_{n\to\infty}\inf_{x,E^n(x)=x}\frac 1n\phi^n_x=\phi_{\text{min}}.\]
\end{proof}
\begin{lemma}
\[F(\beta)\underset{\beta\to+\infty}{\longrightarrow}-\phi_{\text{min}}.\]
\end{lemma}
\begin{proof}
Let $\beta>0$.  Let us write
\[F_n(\beta) = \frac {1} {n \beta}\log\left(\sum_{E^n(x)=x}e^{-\beta\phi_x^n}\right),\]
so that 
\[F(\beta) \underset{(\ref{def_Pr},\ref{def_F})}{=} \lim_{n\to\infty}F_n(\beta).\]
Let  $\epsilon>0$.
By definition of $\phi_{\min}$, for $n$ large enough, 
\[\forall x\in\mathrm{Per}(n), \phi^n_x\geq n(\phi_{\min}-\epsilon)\]
and
\[\exists x\in\mathrm{Per}(n), \phi^n_x\leq n(\phi_{\min}+\epsilon).\]
Thus,
\[e^{-\beta n(\phi_{\min}+\epsilon)}\leq\sum_{E^n(x)=x}e^{-\beta\phi_x^n}\leq l^n e^{-\beta n(\phi_{\min}-\epsilon)}\]
and consequently
\[-\phi_{\min}-\epsilon\leq F_n(\beta)\leq \frac{\log l}{\beta}-\phi_{\min}+\epsilon.\]
Hence, letting $\epsilon\to 0$, we get
\[-\phi_{\min} \leq F(\beta)\leq \frac{\log l}{\beta}-\phi_{\min}.\]
When $\beta$ goes to infinity, the result follows.
\end{proof}

\begin{proof}[Proof of Lemma \ref{ref}]
Now we take $\phi = J = \log( E')$.  By the definition of $J_{\min}$
\[\inf_{O\in\mathrm{Per}(n)}{J^n_O}=nJ_{\min}+o(n),\]
thus
\begin{equation}
    \label{sE1}
\sup_{O\in\mathrm{Per}(n)}\frac{1}{e^{J_O}-1}=e^{-nJ_{\min}+o(n)}=e^{n\lim_{\beta\to\infty} F(\beta)+o(n)}.
\end{equation}
We have
\begin{multline}
\label{BG4}\frac 1{\sqrt n}\left(\sum_{m|n}m\sum_{O\in\mc P_m}\frac 1{(e^{J_O^n}-1)^2}\right)^{-\frac 12}\leq A_n \ \ \ \underset{(\ref{def_An})}{=} \left(\sum_{m|n}m^2\sum_{O\in\mc P_m}\frac 1{(e^{J_O^n}-1)^2}\right)^{-\frac 12}\\\leq\left(\sum_{m|n}m\sum_{O\in\mc P_m}\frac 1{(e^{J_O^n}-1)^2}\right)^{-\frac 12}.\end{multline}
Since 
\[\begin{split}
    \sum_{m|n}m\sum_{O\in\mc P_m}\frac 1{(e^{J_O^n}-1)^2}&= \sum_{E^n(x)=x}\frac{1}{(e^{J_x^n}-1)^2}\\
    &=\sum_{E^n(x)=x}e^{-2J_x^n}\left(1+O\left(e^{-J^n_x}\right)\right)\\
    &=\left(\sum_{E^n(x)=x}e^{-2J_x^n}\right)(1+o(1))\\
    &=e^{n\mathrm{Pr}(-2J)+o(n)},
\end{split}\]
Eq.(\ref{BG4}) gives
\begin{equation*}
   \frac 1{\sqrt n}  e^{-\frac{n}{2} \mathrm{Pr}(-2J)+o(n)} \leq A_n \leq  e^{-\frac{n}{2} \mathrm{Pr}(-2J)+o(n)}
\end{equation*}
hence from Remark \ref{absorption}
\[nA_n= e^{-\frac n2\mathrm{Pr}(-2J)+o(n)}=e^{-nF(2)+o(n)}.\]

Finally, 
\[nA_n\sup_{O\in\mathrm{Per}(n)}\frac{1}{e^{J_O}-1} \underset{(\ref{sE1})}{=} e^{n(\lim\limits_\infty F - F(2)+o(n))}\to 0\]
from Corollary \ref F.
\end{proof}
\bibliographystyle{amsalpha}

\newpage
\bibliography{bib}

\end{document}